\documentclass[12pt,reqno]{amsart}
\usepackage[all,2cell]{xy} \UseAllTwocells \SilentMatrices

\usepackage{amssymb}
\newtheorem{theorem}{Theorem}[section]
\newtheorem{proposition}[theorem]{Proposition}
\newtheorem{lemma}[theorem]{Lemma}
\newtheorem{corollary}[theorem]{Corollary}
\newtheorem{question}[theorem]{Question}

\theoremstyle{definition}
\newtheorem{definition}[theorem]{Definition}
\newtheorem{remark}[theorem]{Remark}

\numberwithin{equation}{section}

\begin{document}
\baselineskip=15.5pt

\title[Tannakian subcategories of integrable holomorphic connections]{On certain Tannakian categories
of integrable connections over K\"ahler manifolds}

\author[I. Biswas]{Indranil Biswas}

\address{School of Mathematics, Tata Institute of Fundamental
Research, Homi Bhabha Road, Mumbai 400005, India}

\email{indranil@math.tifr.res.in}

\author[J. P. dos Santos]{Jo\~ao Pedro dos Santos}

\address{Institut de Math\'ematiques de Jussieu -- Paris Rive Gauche, 4 place Jussieu,
Case 247, 75252 Paris Cedex 5, France}

\email{joao\_pedro.dos\_santos@yahoo.com}

\author[S. Dumitrescu]{Sorin Dumitrescu}

\address{Universit\'e C\^ote d'Azur, CNRS, LJAD, France}

\email{dumitres@unice.fr}

\author[S. Heller]{Sebastian Heller}

\address{Mathematisches Institut, Ruprecht-Karls-Universit\"at Heidelberg,
Im Neuenheimer Feld 205, 69120 Heidelberg, Germany}

\email{seb.heller@gmail.com}



\date{}

\begin{abstract}
Given a compact K\"ahler manifold $X$, it is shown that pairs of the form $(E,\, D)$, where
$E$ is a trivial holomorphic vector bundle on $X$, and $D$ is an integrable holomorphic connection
on $E$, produce a neutral Tannakian category. The corresponding pro-algebraic affine group scheme
is studied. In particular, it is shown that this pro-algebraic affine group scheme for
a compact Riemann surface determines uniquely the isomorphism class of the Riemann surface.
\end{abstract}

\maketitle

\tableofcontents

\section{Introduction}\label{se1}

A question of Ghys asks the following: Is there a pair of the form $(M,\, D)$, where $M$ is a compact Riemann 
surface of genus at least two, and $D$ is an irreducible holomorphic $\text{SL}(2, {\mathbb C})$--connection on 
the rank two trivial holomorphic vector bundle ${\mathcal O}^{\oplus 2}_M$, such that the image of the 
monodromy homomorphism for $D$ is contained in a cocompact lattice of $\text{SL}(2, {\mathbb C})$. The 
motivation for this question comes from the study of compact quotients of $\text{SL}(2, {\mathbb C})$ by 
lattices. Such quotients are compact non-K\"ahler manifolds. While they can't contain a complex
surface \cite[p.~239, Theorem 2]{HM}, it is not known whether they can contain 
compact Riemann surfaces of genus $g \,>\,1$. A positive answer to Ghys' question would provide a nontrivial 
holomorphic map from the Riemann surface $M$ to the quotient of $\text{SL}(2, {\mathbb C})$ by the cocompact
lattice containing the image of the monodromy homomorphism for $D$. In fact the two
problems are equivalent (see \cite{CDHL} for explanations for the origin of Ghys' question).

Let us also mention that a related question of characterizing rank two holomorphic vector bundles $\mathcal V$ 
over a compact Riemann surface, such that for some holomorphic connection on $\mathcal V$ the image of the 
associated monodromy homomorphism is Fuchsian, was raised in \cite[p.~556]{Ka} (there this
question is attributed to Bers).

The above, still open questions of Bers and Ghys, and some related questions in \cite{CDHL, Ka}, motivated us 
to investigate the holomorphic connections on a trivial holomorphic vector bundle. Answering a question asked in 
\cite{CDHL}, examples of irreducible holomorphic $\text{SL}(2, {\mathbb C})$--connection with
real monodromy, on the trivial holomorphic 
$\text{SL}(2, {\mathbb C})$--bundle over a compact Riemann surface, were constructed in \cite{BDH}.

Here we consider integrable holomorphic connections on trivial holomorphic vector bundles over a compact K\"ahler 
manifold $X$. The purpose of the above discussions is to demonstrate the significance of holomorphic connections on 
the trivial holomorphic vector bundles. While highly motivated by the above mentioned questions of Bers and Ghys, 
it should be clarified that the present work does not shed particular light on those questions which remain open, 
but provides a study of the relevant category of integrable holomorphic connections on trivial holomorphic vector 
bundles.

Once we fix a base point $x_0\, \in\, X$ in order
to define a fiber functor, 
using the Tannakian category theory it is shown that the category of
integrable holomorphic connections on trivial holomorphic vector bundles produces a quotient of 
the pro-algebraic completion $\varpi(X,\, x_0)$ of the fundamental group $\pi_1(X,\, x_0)$ (the
details are in Section \ref{se3});
this quotient of $\varpi(X,\, x_0)$ is denoted by ${\Theta}(X,\, x_0)$. Then we prove a 
Torelli type Theorem with respect to ${\Theta}(X,\, x_0)$, for compact Riemann surfaces and also for compact 
complex tori.

The main results of Section \ref{se3} and Section \ref{se4} are the following:
\begin{enumerate}
\item \textit{For compact K\"ahler manifolds $X$ and $Y$,
the natural homomorphism $$\Theta (X,\,x_0)\times\Theta (Y,\,y_0)\,\longrightarrow
\,\Theta (X\times Y,\, (x_0,\,y_0))$$ is an isomorphism.} (See Proposition \ref{prop1}.)

\item \textit{Let $\beta\, : \, X\, \longrightarrow\, Y$ be an orientation preserving diffeomorphism
between compact Riemann surfaces such that the corresponding homomorphism
$$
\beta_\natural\, :\, \varpi(X,\, x_0)\, \longrightarrow\, \varpi(Y,\, \beta(x_0))
$$
descends to a homomorphism from ${\Theta}(X,\, x_0)$ to ${\Theta}(Y,\, \beta(x_0))$. Then the
two Riemann surfaces $X$ and $Y$ are isomorphic.} (See Theorem \ref{thm1}.)

\item \textit{Let $\varphi\, :\, {\mathbb T}\, \longrightarrow\, {\mathbb S}$ be a diffeomorphism
between two compact complex tori such that the corresponding homomorphism
$$
\varphi_*\, :\, \varpi({\mathbb T},\, x_0)\, \longrightarrow\,
\varpi({\mathbb S},\, \varphi(x_0))
$$
descends to a homomorphism from ${\Theta}({\mathbb T},\, x_0)$ to ${\Theta}({\mathbb S},\, \varphi(x_0))$.
Then there is a biholomorphism
$$
{\mathbb T}\, \longrightarrow\, {\mathbb S}
$$
which is homotopic to the map $\varphi$.} (See Proposition \ref{prop2}.)
\end{enumerate}

In Section \ref{lastsec} we consider integrable holomorphic connections on
holomorphic vector bundles over $X$ which decompose into a 
direct sum of holomorphic line bundles. Using the Tannakian category theory in a similar way, we show that 
this category also produces a quotient ${\Delta}(X,\, x_0)$ of the pro-algebraic completion $\varpi(X,\, x_0)$ 
of the fundamental group $\pi_1(X,\, x_0)$. Then we adapt our methods in Section \ref{se4}
in order to prove the same Torelli type theorems for ${\Delta}(X,\, x_0)$.

As is well-known, the category of vector bundles easily fails to be abelian, but as observed as far back as 
\cite[Proposition 3.1]{seshadri67}, semi-stability can be brought in to mend this failure and produce interesting 
{\it abelian} categories of vector bundles. Section \ref{07.05.2020--1} connects the group scheme $\Theta$ to such 
an idea by means of the category of pseudostable vector bundles (stemming from Simpsons's foundational work). We 
produce group schemes $\Sigma(X,\,x_0)$, accounting for connections on pseudostables, and $\pi^{\rm S}(X,\,x_0)$, 
accounting for pseudostables solely, and then establish a link between these and $\Theta$; see Theorem 
\ref{04.05.2020--5}.

\section{Neutral Tannakian categories for a K\"ahler manifold}\label{se2}

Let $X$ be a compact connected K\"ahler manifold of complex dimension $d$. Fix a K\"ahler form $\omega$ on $X$.
The \textit{degree} of a torsionfree coherent analytic sheaf $F$ on $X$ is defined to be
$$
\text{degree}(F)\, :=\, (c_1(\det F)\cup\omega^{d-1})\cap [X]\, \in\, \mathbb R\, . 
$$
Here and elsewhere, we write $c_i$ for the $i$--th Chern class in $H^{2i}_{dR}(X,\,\mathbb R)$ and
define the holomorphic line bundle $\det F$ following \cite[Ch.~V, \S~6]{Ko}. The real number
$$
\mu(F)\,:=\, \frac{\text{degree}(F)}{\text{rank}(F)}
$$
is called the \textit{slope} of $F$.

Fix a base point $x_0\, \in\, X$. Let
\begin{equation}\label{e5}
\phi\, :\, \pi_1(X,\, x_0) \, \longrightarrow\, \varpi(X,\, x_0)
\end{equation}
be the pro-algebraic completion of the fundamental group $\pi_1(X,\, x_0)$. We recall that 
$\varpi(X,\, x_0)$ is a pro-algebraic affine group scheme over $\mathbb C$ which is uniquely 
characterized by the following property: for any homomorphism
$$
\gamma\, :\, \pi_1(X,\, x_0)\, \longrightarrow\, G
$$
to a complex affine algebraic group $G$, there is a unique algebraic homomorphism
$$
\widehat{\gamma}\, :\, \varpi(X,\, x_0)\, \longrightarrow\, G
$$
such that $\widehat{\gamma}\circ\phi\,=\, \gamma$. There are mainly two equivalent constructions of 
$\varpi(X,\,x_0)$: one by means of the Tannakian category of finite dimensional representations of 
$\pi_1(X,\,x_0)$ and Tannakian duality \cite{DMOS} and the other by Freyd's adjoint functor theorem applied to the 
$\mathbb C$--points functor from group schemes to groups \cite[p.~84, 3.J]{freyd64}.

We shall recall from \cite[p.~70]{Si2} three neutral Tannakian categories associated to
the pointed K\"ahler manifold $(X,\, x_0)$ which, in particular, furnish a
Tannakian description of $\varpi(X,\, x_0)$.

Let ${\mathcal C}_{dR}(X)$ denote the category whose objects are pairs of the form
$(E,\, D)$, where $E$ is a holomorphic vector bundle on $X$ and $D$ is an integrable
holomorphic connection on $E$ (see \cite{At} for holomorphic connections). Morphisms from
$(E,\, D)$ to $(E',\, D')$ are all holomorphic homomorphisms of vector bundles
$h\, :\, E\, \longrightarrow\, E'$ that intertwine the connections $D$ and $D'$, meaning $D'\circ h\,=\,
(h\otimes{\rm Id}_{\Omega^1_X})\circ D$ as differential operators from $E$ to
$E'\otimes\Omega^1_X$ with $\Omega^1_X$ being the
holomorphic cotangent bundle of $X$. This category is equipped with the operators of direct sum, tensor
product and dualization. More precisely, ${\mathcal C}_{dR}(X)$ is a rigid abelian tensor category
(see \cite[p.~118, definition 1.14]{DMOS} for rigid abelian tensor categories).

\begin{remark}\label{remb}
It is known that any coherent analytic sheaf on $X$ admitting a holomorphic connection is locally
free (see \cite[p.~211, Proposition 1.7]{Bo}); it should be clarified that although
this proposition in \cite{Bo} is stated only for
${\mathcal O}_X$--modules with an integrable holomorphic connection (same as a ${\mathcal D}_X$--module),
its proof uses only the Leibniz rule which is valid for holomorphic connections.
\end{remark}

It is straightforward to check that the above category ${\mathcal C}_{dR}(X)$,
equipped with the faithful fiber functor that sends any object $(E,\, D)$ to the
fiber $E|_{x_0}$ over $x_0\, \in\, X$, defines a neutral Tannakian category (see \cite[p.~138,
Definition 2.19]{DMOS}, \cite{Sa}, \cite[p.~67]{Si2}, \cite[p.~76]{No} for neutral Tannakian category).
Given any neutral Tannakian category, a theorem of Saavedra Rivano associates
to it a pro-algebraic affine group scheme over $\mathbb C$ \cite[p.~130,
Theorem 2.11]{DMOS} (and the remark following \cite[p.~138,
Definition 2.19]{DMOS}), \cite{Sa}, \cite[p.~77, Theorem 1.1]{No}, \cite[p.~69]{Si2}.
Therefore, the neutral Tannakian category ${\mathcal C}_{dR}(X)$
corresponds to a pro-algebraic affine group scheme over $\mathbb C$.

Let ${\mathcal C}_{B}(X)$ denote the category whose objects are all finite dimensional complex representations 
of $\pi_1(X,\, x_0)$. Using the tautological fiber functor, it defines a neutral Tannakian category. The
pro-algebraic affine group scheme over $\mathbb C$ corresponding to ${\mathcal C}_{B}(X)$, by the
above mentioned theorem of Saavedra Rivano (\cite{Sa}, \cite[p.~130,
Theorem 2.11]{DMOS}), is the group scheme $\varpi(X,\, x_0)$ in \eqref{e5} \cite[p.~69, Lemma 6.1]{Si2}.

A Higgs bundle on $X$ is a pair of the form $(E,\, \theta)$, where $E$ is a holomorphic vector bundle
of $X$ and $\theta\, \in\, H^0(X,\, \text{End}(E)\otimes\Omega^1_X)$ with $\theta\bigwedge\theta\,=\, 0$
\cite{Si1}, \cite{Si2}; the holomorphic section $\theta$ is called a
\textit{Higgs field} on $E$. A Higgs bundle $(E,\, \theta)$ is called \textit{stable} (respectively,
\textit{semistable}) if
$$
\mu(F)\, <\, \mu(E) \ \ \text{(respectively, }\,~\mu(F)\, \leq\, \mu(E)\text{)}
$$
for every coherent analytic subsheaf $F\, \subset\, E$ with $0\, <\, \text{rank}(F)\, <\, \text{rank}(E)$
and $$\theta(F)\, \subset\, F\otimes\Omega^1_X\, .$$ A Higgs bundle $(E,\, \theta)$ is called \textit{polystable}
if it is a direct sum of stable Higgs bundles of same slope.
 
Let ${\mathcal C}_{Dol}(X)$ denote the category whose objects are Higgs bundles $(E,\, \theta)$
such that $ch_2(E)\cup \omega^{d-2}\,=\, 0$ and there is a filtration of holomorphic subbundles
$$
0\, =\, E_0\, \subset\, E_1\,\subset\, \cdots \, \subset\, E_i \,\subset\,\cdots\, \subset\,
E_{\ell-1} \,\subset\, E_\ell\,=\, E
$$
satisfying the following three conditions:
\begin{enumerate}
\item{} $\theta(E_j)\, \subset\, E_j\otimes\Omega^1_X$ for all $1\, \leq\, j\, \leq\, \ell$,

\item $\text{degree}(E_j/E_{j-1})\,=\, 0$ for all $1\, \leq\, j\, \leq\, \ell$, and 

\item the Higgs bundle $(E_j/E_{j-1}, \, \theta)$ is stable
for all $1\, \leq\, j\, \leq\, \ell$ (with a mild abuse abuse of notation, the Higgs field on $E_j/E_{j-1}$
induced by $\theta$ is also denoted by $\theta$).
\end{enumerate}

Notice that the above second condition implies that $\text{degree}(E)\,=\,0$.

In \cite{BG} such Higgs bundles are called pseudostable. 

A homomorphism from $(E,\, \theta)$
to $(E',\, \theta')$ is a holomorphic homomorphism
$$
h\, :\, E\, \longrightarrow\, E'
$$
such that $\theta'\circ h\,=\, (h\otimes{\rm Id}_{\Omega^1_X})\circ\theta$ as homomorphisms from
$E$ to $E'\otimes\Omega^1_X$. It is known that ${\mathcal C}_{Dol}(X)$
is a rigid abelian tensor category \cite[p.~70]{Si2}. This category ${\mathcal C}_{Dol}(X)$ admits
the faithful fiber functor that sends any object $(E,\, \theta)$ to the fiber $E|_{x_0}$. In other words,
${\mathcal C}_{Dol}(X)$ is a neutral Tannakian category.

The two categories ${\mathcal C}_{dR}(X)$ and ${\mathcal C}_{B}(X)$ are equivalent by the
Riemann--Hilbert correspondence that assigns to a flat connection the corresponding monodromy
representation. Using fundamental theorems of Corlette, \cite{Co}, and Simpson, \cite{Si1},
in \cite{Si2} Simpson proved that
the category ${\mathcal C}_{Dol}(X)$ is equivalent to ${\mathcal C}_{dR}(X)$ (see \cite[p.~36,
Lemma 3.5]{Si2}, \cite[p.~70]{Si2}).
Therefore, all these three neutral Tannakian categories, namely ${\mathcal C}_{dR}(X)$,
${\mathcal C}_{B}(X)$
and ${\mathcal C}_{Dol}(X)$, produce the same pro-algebraic affine group scheme over $\mathbb C$
using the theorem of Saavedra Rivano mentioned earlier
(\cite{Sa}, \cite[p.~130, Theorem 2.11]{DMOS}). In other words, each of
these three neutral Tannakian categories produces the
pro-algebraic affine group scheme $\varpi(X,\, x_0)$ in \eqref{e5}.

It should be mentioned that when $X$ is a smooth complex projective variety, and the cohomology class
of the closed form $\omega$ is rational, then the
category ${\mathcal C}_{Dol}(X)$ coincides with the category of semistable Higgs bundles
$(E,\, \theta)$ on $X$ with $ch_2(E)\cup\omega^{d-2}\,=\, 0$ and $\text{degree}(E)\,=\, 0$
\cite[p.~39, Theorem 2]{Si2}.

\section{The Tannakian subcategory $\mathcal T_{dR}$ of ${\mathcal C}_{dR}$}\label{se3}

\subsection{The Tannakian category and the associated group scheme}\label{03.05.2020--4}

As before, $X$ is a compact connected K\"ahler manifold of complex dimension $d$.

\begin{definition}
Let ${\mathcal T}_{dR}(X)$ be the full subcategory of ${\mathcal C}_{dR}(X)$
(defined in Section \ref{se2}) whose objects are
all the couples $(E,\, D)$ satisfying the condition that the holomorphic vector bundle
$E$ is holomorphically trivial.
\end{definition}

Clearly, ${\mathcal T}_{dR}(X)$ is stable under tensor products and duals. The identity object
$({\mathcal O}_X,\, {\rm d})$, where ${\rm d}$ is the de Rham differential, certainly is an object of
${\mathcal T}_{dR}(X)$. In addition ${\mathcal T}_{dR}(X)$ is stable under quotients, as shown by the next proposition.

\begin{proposition}\label{23.04.2020--1}
Let $(E,\,D)\,=\, ({\mathcal O}^{\oplus r}_X,\, D)$ be an object of ${\mathcal T}_{dR}(X)$ and
$$q\,:\,({\mathcal O}^{\oplus r}_X,\,D)\,\longrightarrow\,(E',\,D')$$ an
epimorphism in ${\mathcal C}_{dR}(X)$. Then the holomorphic vector bundle
$E'$ is also trivial.
\end{proposition}

\begin{proof}
We note that $c_1(E')\,=\,0$ because $E'$ carries an integrable connection
\cite[pp.~192--193, Theorem 4]{At}, \cite[p.~141, Proposition]{griffiths-harris78}. Let $r'$
be the rank of $E'$, and introduce ${\rm Gr}(r,r')$, the Grassmann manifold of
$r'$ dimensional quotients of ${\mathbb C}^r$. Note
that the trivial bundle ${\mathcal O}^{\oplus r}_{{\rm Gr}(r,r')}$ comes with a tautological
quotient of rank $r'$, call it
${\mathcal O}^{\oplus r}_{{\rm Gr}(r,r')}\,\longrightarrow\, U$. Then, $q$ induces a morphism $f\,:\,
X\,\longrightarrow\, {\rm Gr}(r,r')$ such that $E'\,=\,f^*U$.
In particular, we have $c_1(f^*(\det U))\,=\, c_1(E')\,=\, 0$. The
fact that $\det(U)$ is ample implies that $f$ is a constant map \cite[p. 177]{griffiths-harris78}. Hence, $E'$
is holomorphically trivial.
\end{proof}

\begin{corollary}
The full subcategory ${\mathcal T}_{dR}(X)$ of ${\mathcal C}_{dR}(X)$ is an abelian subcategory.
\end{corollary}

\begin{proof}
Let $\alpha\,:\,(E,\,D)\,\longrightarrow\,(E_1,\,D_1)$ be an arrow in ${\mathcal T}_{dR}(X)$. Now,
$\mathrm{Image}(\alpha)$ is a sub-connection, call it $(I,\,D_1)$, of $(E_1,\,D_1)$ and we have an
epimorphism $$\alpha\,:\,(E,\,D) \,\longrightarrow\,(I,\,D_1)$$ in ${\mathcal C}_{dR}(X)$. We note that
$I$ is locally free (see Remark \ref{remb}). We conclude
from Proposition \ref{23.04.2020--1} that $(I,\, D_1)$ is an object of
${\mathcal T}_{dR}(X)$. Again from Proposition \ref{23.04.2020--1} it follows that
$\mathrm{coker}(\alpha)$ is an object of ${\mathcal T}_{dR}(X)$ because it is a quotient of an
object of ${\mathcal T}_{dR}(X)$.

To prove that $\mathrm{kernel}(\alpha)$ is in ${\mathcal T}_{dR}(X)$, consider the dual connections
$(E^*,\,D^*)$ and $(E^*_1,\,D^*_1)$ of $(E,\,D)$ and $(E_1,\,D_1)$ respectively. Let
$$
\alpha^*\,:\,(E^*_1,\,D^*_1)\,\longrightarrow\,(E^*,\,D^*)
$$
be the dual of the homomorphism $\alpha$. The above argument gives that $\alpha^*(E^*_1)\, \subset\,
E^*$ is a trivial subbundle, and the quotient $E^*/\alpha^*(E^*_1)$ is also a trivial holomorphic
vector bundle. Hence
$$
\text{kernel}(\alpha)\,=\, (E^*/\alpha^*(E^*_1))^*
$$
is a trivial holomorphic vector bundle. Therefore, $\mathrm{kernel}(\alpha)$,
equipped with the integrable holomorphic connection induced by $D$, is also in ${\mathcal T}_{dR}(X)$.

Hence, the standard criterion for a subcategory to be an abelian subcategory can be applied \cite[Theorem 3.41]{freyd64}.
\end{proof}

Let
\begin{equation}\label{e2}
{\Theta}(X,\, x_0)
\end{equation}
be the affine group scheme over $\mathbb C$ corresponding
to ${\mathcal T}_{dR}(X)$ via \cite[p.~130, Theorem 2.11]{DMOS} by means of the exact
functor $(E,\,D)\,\longmapsto\, E|_{x_0}$.
From Proposition \ref{23.04.2020--1} and the standard criterion \cite[Proposition 2.21, p.139]{DMOS}
we conclude that the natural arrow of group schemes
\begin{equation}\label{e6}
\mathbf{q}_X\, :\, \varpi(X,\, x_0)\, \longrightarrow\, {\Theta}(X,\, x_0)
\end{equation}
is a quotient homomorphism, where $\varpi(X,\, x_0)$ is the group scheme in \eqref{e5}.

Clearly, for any given holomorphic map $f:Y\longrightarrow X$, we obtain tensor functors 
\[
f^\#\,: \, \mathcal C_{dR}(X)\,\longrightarrow\,\mathcal C_{dR}(Y)
\]
and 
\[
f^\#\,:\, \mathcal T_{dR}(X)\,\longrightarrow\,\mathcal T_{dR}(Y)
\]
defined by $(E,\, D)\, \longmapsto\, (f^*E,\, f^*D)$.
If, in addition, we let $y_0\in Y$ be a point which is taken to $x_0$, we derive
homomorphisms of group schemes
\[
f_\natural\,:\,\varpi(Y,\,y_0)\,\longrightarrow\,\varpi(X,\,x_0)
\]
and 
\[
f_\natural:\Theta(Y,\,y_0)\longrightarrow\Theta(X,\,x_0)\, .
\]

Using the equivalence of categories between ${\mathcal C}_{Dol}(X)$ and ${\mathcal C}_{dR}(X)$ mentioned
in Section \ref{se2}, the above subcategory ${\mathcal T}_{dR}(X)$ of ${\mathcal C}_{dR}(X)$ gives
a Tannakian full subcategory of ${\mathcal C}_{Dol}(X)$. It is natural to ask the following:

\begin{question}\label{q1}
What is a direct description of the Tannakian full subcategory of ${\mathcal C}_{Dol}(X)$
corresponding to the subcategory ${\mathcal T}_{dR}(X)$ of ${\mathcal C}_{dR}(X)$?
\end{question}

Answering Question \ref{q1} amounts to classifying all Higgs bundles $(E,\, \theta)$ such that the holomorphic
vector bundle $V$ underlying the flat connection $(V,\, D)$ corresponding to $(E,\, \theta)$ is
trivial. The problem is that the holomorphic structure of $V$ depends on $(E,\, \theta)$ in a rather
intricate manner.

\subsection{Characters of $\Theta$: multiplicative and additive}\label{01.05.2020--2}

Let ${\mathbb G}_m$ be the multiplicative group of nonzero scalars and
${\mathbb G}_a$ the additive group of scalars.
Let $G$ be an affine group scheme over $\mathbb C$. Two of the most basic abstract groups 
associated to it are its group of characters $\mathrm{Hom}(G,\,{\mathbb G}_m)$ and its group 
of additive characters $\mathrm{Hom}(G,\,{\mathbb G}_a)$ (notations are those of \cite[p.~5, 
1.2]{waterhouse79}). These are denoted respectively by $\mathbb X(G)$ and $\mathbb X_a(G)$ in 
what follows. We note that $\mathbb X_a(G)$ comes with the extra structure of a $\mathbb 
C$-vector space. Clearly, $\mathbb X(G)$ is the group of isomorphism classes of rank one 
representations of $G$. In turn, the standard immersion of ${\mathbb G}_a$ into ${\rm GL}(2, 
{\mathbb C})$ permits us to view $\mathbb X_a(G)$ as the vector space $$\mathrm{Ext}_G( 
{\mathbf 1},\, {\mathbf 1})$$ parametrizing extensions of the rank one trivial representation by itself. 

We shall record some observations on characters and additive characters of $\varpi(X,\,x_0)$ 
and $\Theta(X,\,x_0)$. In what follows,
\[
\mathrm{Pic}^\tau(X)
\]
is the subgroup of $\mathrm{Pic}(X)$ formed by the classes of
holomorphic line bundles on $X$ with vanishing first Chern class. 
Recall that according to our convention, Chern classes are real. The subgroup $\mathrm{Pic}^0(X)$ of 
classes having vanishing {\it integral} first Chern class is of finite index in ${\rm Pic}^\tau(X)$ as we 
learn from the theorem of the base.

\begin{lemma}\label{05.05.2020--1}
The map 
\[
H^0(X,\,\Omega_X^1)\,\longrightarrow\, {\mathbb X}(\Theta(X,\,x_0))
\]
sending a form $\alpha$ to the isomorphism class $({\mathcal O_X},\,\mathrm d+\alpha)$ is an isomorphism. The
group $\mathbb X(\varpi(X,\, x_0))$ sits in a short exact sequence 
\[
1\,\longrightarrow\, {\mathbb X}(\Theta(X,\,x_0))\,\longrightarrow\,{\mathbb X}(\varpi(X,\,x_0))
\,\longrightarrow\, \mathrm{Pic}^\tau(X)\,\longrightarrow\,1\, .
\]
\end{lemma}

\begin{proof}The first claim is very simple and its verification is omitted. Then, if $(L,\,\nabla)$ is an 
integrable connection with $L$ a holomorphic line bundle, the interpretation of $c_1(L)$ as curvature form 
\cite[p.~141, Proposition]{griffiths-harris78} assures that $L\,\in\,\mathrm{Pic}^\tau(X)$, and we obtain the arrow 
$$\mathbb X(\varpi(X,\,x_0))\,\longrightarrow\,\mathrm{Pic}^\tau(X)\,.$$ The kernel is precisely the group of 
(isomorphism classes of) objects in $\mathcal T_{dR}(X)$ of rank one, which is $\mathbb X(\Theta(X,\,x_0))$. That 
any class in $\mathrm{Pic}^\tau(X)$ carries an integrable connection is explained by \cite[p.~40,
Corollary 1.3.12]{LT}.
\end{proof}

We shall now be concerned with the additive characters or, with the vector spaces 
\[
{\mathbb X}_a(\varpi(X,\,x_0)) \,\simeq\, \mathrm{Ext}_{\mathcal C_{dR}(X)}({\mathbf 1},\,{\mathbf 1})
\]
and 
\[
{\mathbb X}_a(\Theta(X,\,x_0)) \,\simeq\, \mathrm{Ext}_{\mathcal T_{dR}(X)}({\mathbf 1},\,{\mathbf 1})\, ,
\]
where ${\mathbf 1}$ is the identity object of the pertinent categories. 

Let $H^1_{dR}(X,\,\mathbb C)$ be $\mathbb H^1(X,\, \Omega_X^\bullet)$ \cite[p.~446]{griffiths-harris78}, which
is of course canonically isomorphic to $H^1(X,\,\mathbb C)$ \cite[p.~448]{griffiths-harris78}, and let 
\begin{equation}\label{la}
\lambda_X\,:\, H^0(X,\,\Omega_X^1)\,\longrightarrow\, H^1_{dR}(X,\,\mathbb C)
\end{equation}
be the natural arrow obtained from the fact that global holomorphic one forms on $X$ are closed. Here
we use that $X$ is K\"ahler.

\begin{lemma}\label{01.05.2020--1}
There are functorial isomorphisms of vector spaces
\[
\mathrm{Ext}_{\mathcal C_{dR}(X)}({\mathbf 1},\,{\mathbf 1})\,
\stackrel{\sim}{\longrightarrow}\,H^1_{dR}(X,\, {\mathbb C})
\]
and 
\[
\mathrm{Ext}_{\mathcal T_{dR}(X)}({\mathbf 1},\,{\mathbf 1})\,
\stackrel{\sim}{\longrightarrow}\, H^0(X,\,\Omega_X^1)\, .
\]
Under these isomorphisms, the canonic arrow 
\[
\mathrm{Ext}_{\mathcal T_{dR}(X)}({\mathbf 1},\,{\mathbf 1}) \,\longrightarrow\, 
\mathrm{Ext}_{\mathcal C_{dR}(X)}({\mathbf 1},\, {\mathbf 1})
\]
corresponds to the homomorphism $\lambda_X$ in \eqref{la}.
\end{lemma}

\begin{proof} 
Let $(E,\,\nabla)\,\in\,{\mathcal C}_{dR}(X)$ be an extension of $(\mathcal O_X,\,\mathrm d)$ by itself. Let
${\mathfrak U}\,=\,\{U_i\}$ be an open covering of $X$ on which we have $e_i\,\in\, H^0(U_i,\,E)$ mapping to
$1\,\in\,\mathcal O_X(U_i)$. Consider
the holomorphic $1$--form $\theta_i$ on $U_i$ such that $\nabla(e_i)\,=\,1\otimes \theta_i$.
Then, on $U_i\cap U_j$, 
we have $e_i\,=\,a_{ij}\cdot1+e_j$ with $a_{ij}$ a holomorphic function. The element
$$(\theta_i,\,a_{ij})\,\in\, C^0(\mathfrak U,\,\Omega_X^1)\oplus C^1(\mathfrak U,\,\mathcal O_X)$$
defines a $1$-cocycle of the total \v{C}ech complex 
associated to the bicomplex $$(C^p(\mathfrak U,\,\Omega^q_X))_{0\le p,q}:$$
the form $\theta_i$ is closed and ${\rm 
d}a_{ij}\,=\,\theta_i-\theta_j$. Hence, $\{(\theta_i,\,a_{ij})\}$ gives an element of $H^1_{dR}(X,\,\mathbb C)$, call it 
$[E,\,\nabla]$. Note that a 1-cocycle with coefficients in $\mathrm{GL}_2$ representing the vector bundle $E$ is 
given by the matrices $\begin{pmatrix}1&-a_{ij}\\0&1\end{pmatrix}$ so that $E$ is isomorphic to ${\mathcal 
O}_X^2$ if and only if there exists, after eventually passing to a finer covering, a 0--cochain $(h_i)$ with 
$$h_j-h_i\,=\,-a_{ij}\, .$$ In this case, $[E,\,\nabla]$ belongs to the subspace $H^0(X,\,\Omega_X^1)\,\subset\, 
H^1_{dR}(X,\,\mathbb C)$.

It is a lengthy but straightforward verification to show that 
\[
\mathrm{Ext}_{\mathcal C_{dR}(X)}({\mathbf 1},\,{\mathbf 1})\,\longrightarrow\,
H_{dR}^1(X,\,\mathbb C)\, ,\ \ (E,\,\nabla)\,\longmapsto\,[E,\,\nabla]
\]
is actually bijective.
\end{proof}

\subsection{Properties of the pro-algebraic group scheme $\Theta$}

\begin{lemma}\label{lem2}
Let $X$ and $Y$ be compact connected K\"ahler manifolds and $V$ a holomorphic vector bundle on $X\times Y$
of rank $r$ such that
\begin{itemize}
\item the restriction of $V$ to $\{x\}\times Y$ is holomorphically isomorphic to
${\mathcal O}^{\oplus r}_Y$ for every $x\, \in\, X$, and 

\item the restriction of $V$ to $X\times\{y\}$ is holomorphically isomorphic to
${\mathcal O}^{\oplus r}_X$ for every $y\, \in\, Y$.
\end{itemize}
Then $V$ is holomorphically isomorphic to ${\mathcal O}^{\oplus r}_{X\times Y}$.
\end{lemma}

\begin{proof}
Fix a point $y_0\, \in\, Y$. For any point $x\, \in\, X$, the restriction $V_x$ of $V$ to
$\{x\}\times Y$ trivializable. Hence the evaluation to $y_0$ of sections of $V_x$ 
$$
H^0(Y, \, V_x) \, \longrightarrow\, V_{x,y_0}\, ,\ \ \sigma\, \longmapsto\, \sigma(y_0)
$$
is an isomorphism. This implies that for the natural projection $$p^X\, :\, X\times Y\,
\longrightarrow\, X\, ,$$ the direct image of $V$ to $X$ is
\begin{equation}\label{e1}
p^X_* V\,=\, V^{y_0}\, :=\, V\vert_{X\times\{y_0\}}\, .
\end{equation}
The vector bundle $V^{y_0}$ is supposed to be holomorphically trivial. Fixing $r$ linearly
independent holomorphic sections $\{s_1,\, \cdots, \, s_r\}$ of $V^{y_0}$, we get a
holomorphic trivialization of $V^{y_0}$. Using the isomorphism in \eqref{e1}, each
$s_i$ produces a section
$$
\widetilde{s}_i\, \in\, H^0(X, \, p^X_* V)\,=\, H^0(X\times Y,\, V)\, .
$$
Now these holomorphic sections $\{\widetilde{s}_1,\, \cdots, \, \widetilde{s}_r\}$
of $V$ trivialize $V$ holomorphically.
\end{proof}

\begin{proposition}The group scheme $\Theta(X,\,x_0)$ is the projective limit of
connected algebraic groups.
\end{proposition}

\begin{proof}
Consider an algebraic quotient $\Theta(X,\,x_0)\, \longrightarrow\,
\Theta'$. We need to show that $\Theta'$ is connected and
this amounts to showing that if $\Theta'\, \longrightarrow\, G$ is an algebraic quotient morphism
to a finite group $G$, then $G$ is the trivial group. Now, write $R$ for the
left regular representation of $G$ and note that the multiplication
operation $R\otimes R\, \longrightarrow\,
R$ and the identity $\mathbb C\, \longrightarrow\, R$ are $G$--equivariant.

Since $G$ is a quotient of $\Theta(X,\,x_0)$, any $G$--module gives a $\Theta(X,\,x_0)$--module.
In particular, $R$ gives a $\Theta(X,\,x_0)$--module. Let $({\mathcal R},\, \nabla)$ be the
object in the category ${\mathcal T}_{dR}(X)$ corresponding to this $\Theta(X,\,x_0)$--module $R$.
We note that ${\mathcal R}$ is the trivial holomorphic vector
${\mathcal O}_X\otimes_{\mathbb C}R$ on $X$ with fiber $R$. In particular, the rank of ${\mathcal R}$ is $|G|$.
The above homomorphism $\mathbb C\, \longrightarrow\, R$ produces an arrow
${\mathcal O}_X \, \longrightarrow\, {\mathcal R}$. The above mentioned multiplication
operation $R\otimes R\, \longrightarrow\,
R$ produces another arrow ${\mathcal R}\otimes
{\mathcal R}\, \longrightarrow\, {\mathcal R}$.
Endowed with these, $\mathcal R$ becomes a locally free ${\mathcal
O}_X$--algebra. Note, in addition, that giving ${\mathcal R}\otimes{\mathcal R}$ the
tensor product connection, the arrow ${\mathcal R}\otimes {\mathcal R}
\, \longrightarrow\, {\mathcal R}$ is also horizontal.

Let $f\,:\,Y\, \longrightarrow\, X$ be the analytic spectrum of $\mathcal R$ (see
\cite[Expos\'e 19, Definition 2]{seminaire_cartan} and \cite[1.14-15]{fischer}); the morphism $f$ is
finite and $f_*\mathcal
O_Y\,=\,\mathcal R$. In addition, it is not hard to see that
$f$ is a {\it local biholomorphism} so that $Y$ is a compact complex {\it manifold}; here are the details.
Let $x\,\in\, X$ be arbitrary. We have a $G$--equivariant isomorphism of $R$-algebras
$$R\otimes R\, \longrightarrow\,R^{\oplus|G|}$$ which allows us to see that the ${\mathcal
R}_{x}$--algebra ${\mathcal R}_x\otimes_{\mathcal O_{X,x}}{\mathcal R}_x$
(multiplication on the left) is isomorphic to ${\mathcal R}_x^{\oplus |G|}$.
Consequently, ${\mathcal O}_{X,x}\, \longrightarrow\, \mathcal R_x$ is an \'etale morphism by flat
descent \cite[II, Proposition 4, p. 14]{raynaud}. 
Let now $y\in Y$ be above $x$. We know that $\mathcal O_{Y,y}$ is, as an $\mathcal O_{X,x}$--algebra, the localization of $\mathcal R_x$ at some maximal ideal \cite[Expos\'e 19, Proposition 6]{seminaire_cartan}. So, the natural arrow $\mathcal O_{X,x}\, \longrightarrow\,\mathcal O_{Y,y}$ is an isomorphism \cite[VII, Proposition 3, p. 76]{raynaud}. 

Finally, $Y$ is connected since any idempotent $$e\,\in\, H^0(X,\,\mathcal R)$$ is
horizontal so that it produces a horizontal arrow ${\mathcal O}_X\, \longrightarrow\, {\mathcal R}$
and then a $G$--invariant element of $R$.

Because
\[\begin{split}1&\,=\,\dim H^0(Y,\,\mathcal O_Y)\\
&\,=\,\dim H^0(X,\,\mathcal R)\\
&\,=\, |G|\cdot\dim H^0(X,\,\mathcal O_X),
\end{split}
\]
we conclude that $|G|\,=\,1$.
\end{proof}

\begin{proposition}\label{prop1}
Let $(X,\,x_0)$ and $(Y,\,y_0)$ be pointed compact K\"ahler manifolds. Set
$P\,=\,X\times Y$, and let $i\,:\,X\,\longrightarrow\, P$ (respectively, $j\,:\,Y
\,\longrightarrow\, P$) stand for the immersion defined by $x\, \longmapsto\, (x,\, y_0)$
(respectively, $y\, \longmapsto\, (x_0,\, y)$).
Then, the arrow of affine group schemes 
\[(i_\natural,\,j_\natural)\,:\,\Theta(X,\,x_0)\times\Theta(Y,\,y_0)\,\longrightarrow
\,\Theta(P,\, (x_0,\,y_0))
\]
is an isomorphism. 
\end{proposition}

\begin{proof}
Let $p$ (respectively, $q$) be the natural projection of $P$ to $X$ (respectively, $Y$).
In what follows, we abandon reference to the base points. 

We know that the natural morphism 
\[(i_\natural,\,j_\natural)\,:\,\varpi(X)\times\varpi(Y)\,\longrightarrow \,\varpi(P)
\]
is an isomorphism. 

We claim that $i_\natural\,:\,\Theta(X)\,\longrightarrow\,
\Theta(P)$ is a closed and normal immersion.

First note that $i_\natural$ is closed since
$$p_\natural \circ i_\natural\,:\,\Theta(X)\,\longrightarrow\, \Theta(X)$$ is the identity
map. Next, $i_\natural$ is normal since $i_\natural\,:\,\varpi(X)\,\longrightarrow\, \varpi(P)$ is a
normal immersion and we posses the following commutative diagram 
\[
\begin{matrix}
\varpi(X) & \stackrel{i_\natural}{\longrightarrow} &\varpi(P)\\
{\mathbf q}_X \Big\downarrow\,\,\,\,\,\,\, && {\mathbf q}_P\Big\downarrow\,\,\,\,\,\,\, \\
\Theta(X)& \stackrel{i_\natural}{\longrightarrow} & \Theta(P)
\end{matrix}
\]
in which the vertical arrows are the quotient morphisms. This proves the claim.

Clearly, the same arguments apply to $j_\natural\,:\,\Theta(Y)\,\longrightarrow\,\Theta(P)$. 

Let \[\chi\,:\,\Theta(P)\,\longrightarrow\, Q\]
be the cokernel of $i_\natural\,:\,\Theta(X)\,\longrightarrow\, \Theta(P)$
\cite[16.3, Theorem]{waterhouse79}. It is easy to see that the
kernel of the composition
\[
\Theta(Y) \,\stackrel{j_\natural}{\longrightarrow}\,\Theta(P)\,\stackrel{\chi}{\longrightarrow}\, Q 
\]
is trivial; indeed, $q_\natural\,:\,\Theta(P)\,\longrightarrow\,\Theta(Y)$ factors as 
\[
\Theta(P)\,\stackrel{\chi}{\longrightarrow}\, Q\,\stackrel{r}{\longrightarrow}\Theta(Y)
\] 
since $q_\natural\circ i_\natural$ is trivial, and ${\rm id}_{\Theta(Y)}
\,=\,q_\natural\circ j_\natural$. We conclude that ${\rm Im}(i_\natural)\cap {\rm Im}(j_\natural)
\,=\,\{e\}$. Now, a well-known lemma from group theory says that
$$(i_\natural,\,j_\natural)\,:\,\Theta(X)\times\Theta(Y)\,\longrightarrow\, \Theta(P)$$ is a
normal monomorphism of group schemes \cite[I.4.9, p.~48, Proposition 15]{bourbakialgebra}. As
such, it is a closed immersion \cite[15.3, Theorem]{waterhouse79}. Finally,
since $$(i_\natural,\,j_\natural)\,:\,\varpi(X)\times\varpi(Y)\,\longrightarrow\, \varpi(P)$$
is a quotient morphism, we conclude that $$(i_\natural,\,j_\natural)\,:\,
\Theta(X)\times\Theta(Y)\,\longrightarrow\, \Theta(P)$$ is also one.
Hence, $\Theta(X)\times\Theta(Y)\,\simeq\, \Theta(P)$. 
\end{proof}

Take a pointed compact K\"ahler manifold $(X,\, x_0)$. Let
\begin{equation}\label{Ga}
\Gamma\, \subset\, H^0(X,\, \Omega^1_X)^*
\end{equation}
be the image of the homomorphism
$H_1(X,\, {\mathbb Z})\, \longrightarrow\, H^0(X,\, \Omega^1_X)^*$ that sends
any $$\gamma\, \in\, H_1(X,\, {\mathbb Z})$$ to the element of the dual vector space
$H^0(X,\, \Omega^1_X)^*$ defined by the integral
$$
\nu\, \longmapsto\, \int_\gamma \nu\, , \ \ \nu\, \in\, H^0(X,\, \Omega^1_X)\, .
$$
Let $\text{A}(X)\,=\, H^0(X,\, \Omega^1_X)^*/\Gamma$ be the Albanese variety of $X$. Let
\begin{equation}\label{ea}
{\mathcal A}\, :\, X\, \longrightarrow\, \text{A}(X)
\end{equation}
be the Albanese map defined by
$$
{\mathcal A}(x)(\nu)\, =\, \int_{x_0}^x \nu\, ;
$$
this does not depend on the choice of path from $x_0$ to $x$.
We have $H^0(X,\, \Omega^1_X)\,=\, H^0(\text{A}(X),\, \Omega^1_{\text{A}(X)})$, and the
homomorphism
\begin{equation}\label{al}
\alpha\,: \, H^0(\text{A}(X),\, \Omega^1_{\text{A}(X)})\, \longrightarrow\,
H^0(X,\, \Omega^1_X)
\end{equation}
defined by $\xi\, \longmapsto\, {\mathcal A}^*\xi$, for all $\xi\, \in
\, H^0(\text{A}(X),\, \Omega^1_{\text{A}(X)}),$ is an isomorphism.
Hence the homomorphism
\begin{equation}\label{1f}
\overline{\alpha}\, :\, H^1(\text{A}(X),\, {\mathcal O}_{\text{A}(X)})
\,=\, \overline{H^0(\text{A}(X),\, \Omega^1_{\text{A}(X)})}
\, \longrightarrow\, \overline{H^0(X,\, \Omega^1_X)}\,=\,H^1(X,\, {\mathcal O}_X)
\end{equation}
is an isomorphism.

In view of the isomorphism $\alpha$ in \eqref{al} we conclude that there is a natural bijection between 
the holomorphic connections on ${\mathcal O}^{\oplus r}_X$ and the holomorphic connections on 
${\mathcal O}^{\oplus r}_{\text{A}(X)}$; this isomorphism sends a holomorphic connection $D$ on 
${\mathcal O}^{\oplus r}_{\text{A}(X)}$ to the holomorphic connection ${\mathcal A}^*D$ on 
${\mathcal O}^{\oplus r}_X\,=\, {\mathcal A}^* {\mathcal O}^{\oplus r}_{\text{A}(X)}$.

If $D$ is integrable, then ${\mathcal A}^*D$ is integrable. Notice that the converse need not be true in 
general. To construct examples of non-integrable $D$ such that ${\mathcal A}^*D$ is integrable, 
we recall that there are examples of $(M,\, \mathbb{D})$, where $M$ is a compact Riemann surface 
and $\mathbb{D}$ is an irreducible holomorphic connection on ${\mathcal O}^{\oplus 2}_M$ \cite{CDHL}, 
\cite{BD}, \cite{BDH} (any holomorphic connection on a Riemann surface is automatically 
integrable). Since the fundamental group of a compact complex torus $\mathbb T$ is abelian, there 
is no irreducible integrable connection on a rank two bundle over $\mathbb T$. Also, the pullback of a 
reducible connection is reducible. Therefore, for any $(M,\, \mathbb{D})$ as above, the 
corresponding holomorphic connection on the Albanese variety $\text{A}(M)$ is not integrable.

Note that $\Gamma$ in \eqref{Ga} is a quotient of $\pi_1(X,\, x_0)$, because $H_1(X,\, {\mathbb 
Z})$ is a quotient of $\pi_1(X,\, x_0)$, and $\Gamma$ is a quotient of $H_1(X,\, {\mathbb Z})$.

\begin{proposition}\label{q2}
Let $D$ be an integrable holomorphic connection on ${\mathcal O}^{\oplus r}_X$. Then there is an
integrable holomorphic connection $\widetilde{D}$ on ${\mathcal O}^{\oplus r}_{{\rm A}(X)}$ such that
$$
({\mathcal O}^{\oplus r}_X,\, D)\,=\, ({\mathcal A}^*{\mathcal O}^{\oplus r}_{{\rm A}(X)},\,
{\mathcal A}^*\widetilde{D})\, ,
$$
where $\mathcal A$ is the map in \eqref{ea},
if and only if the monodromy representation $$\rho_D\,:\, \pi_1(X,\, x_0)\, \longrightarrow\, 
{\rm GL}(r,{\mathbb C})$$ for $D$ factors through the quotient group $\Gamma$ of $\pi_1(X,\, x_0)$.
\end{proposition}

\begin{proof}
We have $\pi_1({\rm A}(X))\,=\, \Gamma$, and the homomorphism
$$
{\mathcal A}_*\, :\, \pi_1(X,\, x_0)\, \longrightarrow\, \pi_1({\rm A}(X))
$$
induced by the Albanese map $\mathcal A$ in \eqref{ea} actually coincides with the quotient homomorphism
$\pi_1(X,\, x_0)\, \longrightarrow\, \Gamma$. Therefore, if
$$
({\mathcal O}^{\oplus r}_X,\, D)\,=\, ({\mathcal A}^*{\mathcal O}^{\oplus r}_{{\rm A}(X)},\,
{\mathcal A}^*\widetilde{D})
$$
for some integrable holomorphic connection $\widetilde{D}$ on ${\mathcal O}^{\oplus r}_{{\rm A}(X)}$,
then the monodromy representation $\rho_D\,:\, \pi_1(X,\, x_0)\, \longrightarrow\, 
{\rm GL}(r,{\mathbb C})$ for $D$ factors through the quotient group $\Gamma$ of $\pi_1(X,\, x_0)$.

To prove the converse, assume that the monodromy representation $\rho_D\,:\, \pi_1(X,\, x_0)\, 
\longrightarrow\, {\rm GL}(r,{\mathbb C})$
for $D$ does factor through the quotient group $\Gamma$ of 
$\pi_1(X,\, x_0)$. Therefore, the representation
\begin{equation}\label{rd}
\rho'_D\,:\, \Gamma\,=\, \pi_1({\rm A}(X),\, {\mathcal A}(x_0)) \, \longrightarrow\,
{\rm GL}(r,{\mathbb C})
\end{equation}
given by $\rho_D$ produces a pair $(V,\, \widetilde{D})$, where
\begin{itemize}
\item $V$ is a holomorphic vector bundle of rank $r$ on ${\rm A}(X)$,

\item $\widetilde D$ is an integrable holomorphic connection on $V$, and

\item $({\mathcal A}^*V,\, {\mathcal A}^*\widetilde{D})\,=\, ({\mathcal O}^{\oplus r}_X,\, D)$.
\end{itemize}
Consequently, to prove the proposition it suffices to show that
\begin{equation}\label{s}
V\,\simeq\, {\mathcal O}^{\oplus r}_{{\rm A}(X)}\, .
\end{equation}

Consider ${\mathbb C}^r$ as a $\Gamma$--module using the homomorphism $\rho'_D$ (in \eqref{rd})
together with the standard action of ${\rm GL}(r,{\mathbb C})$ on ${\mathbb C}^r$.
Since $\Gamma$ is an abelian group, we contend that the $\Gamma$--module ${\mathbb C}^r$
decomposes as
\begin{equation}\label{nd}
\mathbb C^r\,=\, \bigoplus_{i=1}^m L_i\otimes U_i\, ,
\end{equation}
where $L_i$ (respectively, $U_i$) is a one-dimensional (respectively, unipotent)
representation of $\Gamma$. Indeed, let $S$ be an indecomposable summand 
of the $\Gamma$--module $\mathbb C^r$ and, for any given $\gamma\,\in\, \Gamma$, consider the
decomposition of $S$ into generalized eigenspaces: $\bigoplus_iS_i$. Since every 
$\delta\,\in\,\Gamma$ commutes with $\gamma$, we can say that every $S_i$ is invariant
under $\Gamma$; it follows that $\gamma$ has a single eigenvalue in $S$ (recall
that $S$ is indecomposable). 
Associating to each $\gamma$ the previous eigenvalue, we get a homomorphism $\Gamma\,\longrightarrow\, \mathbb C^*$ and hence
a one dimensional representation $L$. Now, $L^\vee\otimes S$ is indecomposable and has
only one eigenvalue, namely $1$; here $L^\vee$ denotes
the dual of $L$. Simultaneous triangularization --- which follows easily
from the fact that $\Gamma$ is abelian --- now shows
that $L^\vee\otimes S$ is unipotent thus establishing the decomposition in \eqref{nd}.

{}From \eqref{nd} we obtain a decomposition of the connection $(V,\,\widetilde D)$ as 
\begin{equation}\label{nd2}
(V,\,\widetilde D)\,=\,\bigoplus_{i=1}^m {\mathcal L}_i\otimes {\mathcal U}_i\, ,
\end{equation}
where ${\mathcal L}_i$ stands for a rank one integrable connection and ${\mathcal U}_i$
stands for a connection with unipotent monodromy. Note that, since ${\mathcal U}_i$
has unipotent monodromy, each ${\mathcal L}_i$ is a sub-connection of $V$
and hence ${\mathcal A}^*{\mathcal L}_i$ is a sub-connection of ${\mathcal A}^*V$. Using a dualization and Proposition \ref{23.04.2020--1}, we can say that the vector bundle underlying
${\mathcal A}^*{\mathcal L}_i$ is trivial. Since the homomorphism
$$
\text{Pic}^0({\rm A}(X))\, \longrightarrow\, \text{Pic}^0(X)\, ,\ \ L\,\longmapsto\, {\mathcal A}^*L\, ,
$$
where $\mathcal A$ is the map in \eqref{ea}, is an isomorphism, we conclude that 
\begin{equation}\label{27.04.2020--1}
\text{the~ underlying~ holomorphic~ line~ bundle~ of~ ${\mathcal L}_i$~ in~ \eqref{nd2}~ is~ trivial}
\end{equation}
for all $1\, \leq\, i\,\leq\, m$. 

Next, we note that the connection ${\mathcal A}^*{\mathcal U}_i$ is a sub-connection of
${\mathcal A}^*(V)\otimes({\mathcal A}^*{\mathcal L}^\vee_i)$ and hence, applying a
dualization and Proposition \ref{23.04.2020--1}, from \eqref{27.04.2020--1} we conclude that
$${\mathcal A}^*{\mathcal U}_i\,\in\,{\mathcal T}_{dR}(X)\, .$$

The holomorphic vector bundle underlying ${\mathcal U}_i$ will be denoted by $U_i$.
The integrable holomorphic connection on $U_i$ defining ${\mathcal U}_i$ will be denoted by $D_i$.
Let $s$ be the rank of $U_i$.

Since ${\mathcal U}_i$ is an integrable connection with unipotent
monodromy, there is a filtration of holomorphic subbundles
\begin{equation}\label{f}
0\,=\, V_0\, \subset\, V_1\, \subset\, \cdots\, \subset\, V_i\, \subset\, \cdots\, \subset\,
V_{r-1}\, \subset\, V_s\, =\, U_i
\end{equation}
such that
\begin{enumerate}
\item ${\rm rank}(V_i)\,=\, i$ for all $1\, \leq\, i\, \leq\, s$,

\item every quotient $V_i/V_{i-1}$ is ${\mathcal O}_{{\rm A}(X)}$,

\item each $V_i$ is preserved by $D_i$, and

\item for every $1\, \leq\, i\, \leq\, s$, the connection on $V_i/V_{i-1}$ induced by $D_i$ is the
trivial connection on ${\mathcal O}_{{\rm A}(X)}$ given by the de Rham differential.
\end{enumerate}
Let
\begin{equation}\label{c2}
0\,=\, {\mathcal A}^* V_0\, \subset\, {\mathcal A}^*V_1\, \subset\, \cdots\, \subset\,
{\mathcal A}^*V_i\, \subset\, \cdots\, \subset\,
{\mathcal A}^* V_{s-1}\, \subset\, {\mathcal A}^*V_s\, =\, {\mathcal A}^*U_i\,=\,
{\mathcal O}^{\oplus s}_X
\end{equation}
be the filtration of ${\mathcal A}^*U_i\,=\, {\mathcal O}^{\oplus s}_X$ obtained by pulling back
the filtration in \eqref{f} using the map ${\mathcal A}$. We know that
$$
({\mathcal A}^*V_i)/({\mathcal A}^*V_{i-1})\,=\, {\mathcal A}^*(V_i/V_{i-1})\,=\, {\mathcal O}_X
$$
for all $1\, \leq\, i\, \leq\, s$.

We will prove that the filtration in \eqref{c2} splits holomorphically.

For any $1\, \leq\, j\, \leq\, s$, let ${\mathcal S}_j\,=\, {\mathcal O}_X\, \subset\,
{\mathcal O}^{\oplus r}_X$ be the $j$--th factor in the direct sum ${\mathcal O}^{\oplus
s}_X$. Let $\psi_j$ denote the following composition of homomorphisms:
$$
{\mathcal S}_j\,\hookrightarrow\, {\mathcal O}^{\oplus s}_X\,=\, {\mathcal A}^*V_s\, \longrightarrow\,
({\mathcal A}^*V_s)/({\mathcal A}^*V_{s-1})\,=\, {\mathcal A}^*(V_s/V_{s-1})\,=\, {\mathcal O}_X
$$
(see \eqref{c2}), where ${\mathcal A}^*V_s\, \longrightarrow\,
({\mathcal A}^*V_s)/({\mathcal A}^*V_{s-1})$ is the natural quotient map. For some $1\, \leq\, j_0\, \leq\, s$,
the above homomorphism $\psi_{j_0}$ is nonzero. This
$j_0$ is evidently an isomorphism. Hence we conclude that
\begin{equation}\label{a2p}
{\mathcal O}^{\oplus s}_X\,=\, {\mathcal S}_{j_0}\oplus \left(\bigoplus_{j=1,j\not=j_0}^s {\mathcal S}_j\right)
\,=\, {\mathcal A}^*V_s\,=\, {\mathcal A}^*V_{s-1}\oplus {\mathcal O}_X\, ;
\end{equation}
the above direct summand ${\mathcal O}_X\, \subset\, {\mathcal A}^*V_s$ is the image of ${\mathcal S}_{j_0}$
by the identification in \eqref{c2} between ${\mathcal A}^*V_s$ and ${\mathcal O}^{\oplus s}_X$. 
Applying \cite[p.~315, Theorem 2]{At1} to the decomposition in \eqref{a2p} we conclude that
$$
{\mathcal O}^{\oplus (s-1)}_X\,=\, 
\bigoplus_{j=1,j\not=j_0}^s {\mathcal S}_j\,=\, {\mathcal A}^*V_{s-1}
$$
Repeating the above argument, after replacing ${\mathcal O}^{\oplus s}_X$ (respectively, ${\mathcal A}^*V_s$) by
${\mathcal O}^{\oplus (s-1)}_X$ (respectively, ${\mathcal A}^*V_{s-1}$), we conclude that
$$
{\mathcal A}^*V_{s-1}\,=\, {\mathcal A}^*V_{s-2}\oplus {\mathcal O}_X\, .
$$
Now proceeding inductively we conclude that the filtration in \eqref{c2} splits holomorphically.

The homomorphism
$$
H^1({\rm A}(X), \, {\mathcal O}_{{\rm A}(X)})\, \longrightarrow\, H^1(X,\,
{\mathcal O}_X)\, ,\ \ \theta\, \longmapsto\, {\mathcal A}^*\theta
$$
is an isomorphism. Since the extensions of ${\mathcal O}_{{\rm A}(X)}$ by ${\mathcal O}_{{\rm A}(X)}$
are parametrized by the cohomology $H^1({\rm A}(X), \, {\mathcal O}_{{\rm A}(X)})$,
from this and the fact that the filtration in \eqref{c2} splits holomorphically
it follows that the filtration in \eqref{f} splits holomorphically. This proves that
$U_i\,=\, {\mathcal O}^{\oplus s}_{{\rm A}(X)}$.
\end{proof}

Let us now study the case of ``abelian'' K\"ahler manifolds by first recalling certain fundamental facts 
from the theory of affine group schemes.

Let $U$ be an algebraic affine and unipotent group scheme over $\mathbb C$ \cite[8.3]{waterhouse79}: there is 
a closed immersion of $U$ into some group of strict upper triangular matrices. Endowing the nilpotent Lie 
algebra $\mathrm {Lie}(U)$ with its Baker--Campbell--Hausdorff
multiplication, it is known that $$\exp\,:\,\mathrm{Lie}(U)\,\longrightarrow 
\,U$$ is an isomorphism of group schemes \cite[Theorem XVII.4.2, p. 232.]{hochschild81}; in particular, if $U$ is 
in addition a commutative group scheme, then $U\,\simeq\, {\mathbb G}_a^r$ for some
$r$. Moreover, $r\,=\,\dim_{\mathbb C}\mathrm{Hom}(U,\,\mathbb G_a)$
\cite[Theorem 8.4]{waterhouse79}. Said differently,
\[
U\,\simeq\, \mathrm{Hom}(U,\,{\mathbb G}_a)\, .
\]

Still in the topic of affine group schemes, for any given abstract abelian group $\Lambda$, we shall denote by 
$\mathrm{Diag}(\Lambda)$ the diagonalizable group scheme corresponding to the abstract group $\Lambda$ as 
explained in \cite[2.2]{waterhouse79}; on the level of $\mathbb C$-points, ${\rm Diag}(\Lambda)$ is just 
$\mathrm{Hom}(\Lambda,\,{\mathbb C}^\times)$.

\begin{proposition}
Suppose that $\pi_1(X,\,x_0)$ is abelian. Then 
\[
\Theta(X,\,x_0)\,\simeq\, H^0(X,\,\Omega_X^1)\times{\rm Diag}(H^0(X,\,\Omega_X^1))\, .
\] 
\end{proposition}

\begin{proof}
Since $\pi_1(X,\,x_0)$ is abelian, the group scheme $\varpi(X,\,x_0)$ is abelian. This implies that
$\Theta(X,\,x_0)$ is abelian and hence is a product of a unipotent $U$ and a diagonal group
scheme $\mathbb D$ \cite[p.~70, Theorem, 9.5]{waterhouse79}. The arguments made in Lemma \ref{05.05.2020--1} and
Lemma \ref{01.05.2020--1} jointly with the preliminary material on group schemes recalled above
now allow us to explicitly determine $U$ and $\mathbb D$ as wanted. Indeed, the group of characters of $U\times
{\mathbb D}$
(respectively, additive characters) is simply the group of characters of $\mathbb D$ (respectively, additive
characters of $U$) as explained in \cite[Chapter 8]{waterhouse79}, Corollary in 8.3 and Exercise 6.
\end{proof}

\section{Riemann surfaces and compact complex tori}\label{se4}

\subsection{Neutral Tannakian category for a Riemann surface}

Let $X$ and $Y$ be two compact connected Riemann surfaces of common genus $g$, with $g\, \geq\,1$.
Fix a point $x_0\, \in\, X$. Let
\begin{equation}\label{e3}
\beta\, :\, X\, \longrightarrow\, Y
\end{equation}
be a $C^\infty$ orientation preserving diffeomorphism. Let
$$
\widehat{\beta}\, :\, \pi_1(X,\, x_0)\, \longrightarrow\, \pi_1(Y,\, \beta(x_0))
$$
be the homomorphism of fundamental groups induced by $\beta$. It produces a unique algebraic homomorphism
\begin{equation}\label{e4}
\beta_\natural\, :\, \varpi(X,\, x_0)\, \longrightarrow\, \varpi(Y,\, \beta(x_0))
\end{equation}
such that $\beta_\natural\circ\phi\,=\,\phi_Y\circ\widehat{\beta}$, where $\phi$ is the
homomorphism in \eqref{e5} and $$\phi_Y\, :\, \pi_1(Y,\, \beta(x_0))
\, \longrightarrow\, \varpi(Y,\, \beta(x_0))$$ is the similar homomorphism for $Y$.

We will say that the homomorphism $\beta_\natural$ in \eqref{e4} descends to a homomorphism
from ${\Theta}(X,\, x_0)$, constructed in \eqref{e2}, to ${\Theta}(Y,\, \beta(x_0))$ if
there is a homomorphism
$$
\beta'_\natural\, :\, {\Theta}(X,\, x_0) \, \longrightarrow\, {\Theta}(Y,\,\beta(x_0))
$$
such that
$$
\beta'_\natural\circ\mathbf{q}_X\,=\, \mathbf{q}_Y\circ\beta_\natural\, ,
$$
where $\mathbf{q}_X$ is the homomorphism in \eqref{e6} and $$\mathbf{q}_Y\, :\, \varpi(Y,\, 
\beta(x_0))\, \longrightarrow\, {\Theta}(Y,\, \beta(x_0))$$ is the similar homomorphism for 
$Y$, while $\beta_\natural$ is constructed in \eqref{e4}.

\begin{theorem}\label{thm1}
Assume that the homomorphism $\beta_\natural$ in \eqref{e4} descends to a homomorphism
from ${\Theta}(X,\, x_0)$ to ${\Theta}(Y,\, \beta(x_0))$. Then the
two Riemann surfaces $X$ and $Y$ are isomorphic.
\end{theorem}

\begin{proof}
Let $\beta'_\natural\, :\, {\Theta}(X,\, x_0) \, \longrightarrow\, {\Theta}(Y,\,\beta(x_0))$
be a homomorphism such that the diagram
\begin{equation}\label{e7}
\begin{matrix}
\varpi(X,\, x_0)& \stackrel{\beta_\natural}{\longrightarrow}& \varpi(Y,\, \beta(x_0))\\
\mathbf{q}_X \Big\downarrow\,\,\,\,\,\,\, && \mathbf{q}_Y\Big\downarrow\,\,\,\,\,\,\, \\
{\Theta}(X,\, x_0)&\stackrel{\beta'_\natural}{\longrightarrow}& {\Theta}(Y,\,\beta(x_0))
\end{matrix}
\end{equation}
is commutative. Let
\begin{equation}\label{e8}
\begin{matrix}
\mathbb X_a( \Theta (Y,\, \beta(x_0)) )& {\longrightarrow} &
\mathbb X_a( \Theta (X,\, x_0) )\\
\lambda^2\Big\downarrow\,\,\,\, &&\lambda^1 \Big\downarrow\,\,\,\,\\
\mathbb X_a(\varpi(Y,\, \beta(x_0)))& \stackrel{\rho}{\longrightarrow} &
\mathbb X_a( \varpi(X,\, x_0) )
\end{matrix}
\end{equation}
be the corresponding commutative diagram of vector spaces (see Section \ref{01.05.2020--2}).

Now, according to Lemma \ref{01.05.2020--1}, 
$$
\mathbb X_a( \Theta (X,\, x_0) )\,=\, H^0(X,\, \Omega^1_X)
\ \ \text{ and }\ \ \mathbb X_a( \varpi(X,\, x_0) )\,=\, H^1_{dR}(X,\, {\mathbb C})\, .
$$
Similarly, we have
$$
\mathbb X_a( \Theta(Y,\, \beta(x_0)) )\,=\, H^0(Y,\, \Omega^1_Y)
\ \ \text{ and }\ \ \mathbb X_a( \varpi(Y,\, \beta(x_0)) )\,=\, H^1_{dR}(Y,\, {\mathbb C})\,. 
$$
 The homomorphism
$\lambda^1$ in \eqref{e8} is the natural inclusion of $H^0(X,\, \Omega^1_X)$ in
$H^1_{dR}(X,\, {\mathbb C})$ given by the fact that any holomorphic $1$--form on $X$ is closed.
The homomorphism
\begin{equation}\label{r}
\rho\, :\, H^1_{dR}(Y,\, {\mathbb C})\,=\, \mathbb X_a( \varpi(Y,\, \beta(x_0)) )
\, \longrightarrow\, \mathbb X_a( \varpi(X,\, x_0) )\,=\,
H^1_{dR}(X,\, {\mathbb C})
\end{equation}
in \eqref{e8} coincides with the pullback homomorphism
\begin{equation}\label{e9}
\beta^*\, :\, H^1_{dR}(Y,\, {\mathbb C})\, \longrightarrow\,H^1_{dR}(X,\, {\mathbb C})\, ,\ \
c\, \longmapsto\, \beta^*c\, ,
\end{equation}
where $\beta$ is diffeomorphism in \eqref{e3}.

For the isomorphism $\beta^*$ in \eqref{e9}, we have
$$
\beta^*(H^1(Y,\, {\mathbb Z}))\,=\, H^1(X,\, {\mathbb Z})\, .
$$
Furthermore, since the diffeomorphism $\beta$ is orientation preserving, it takes the natural symplectic 
pairing on $H^1_{dR}(Y,\, {\mathbb C})$ defined by
\begin{equation}\label{e10}
c_1\otimes c_2\, \longmapsto\, \int_Y c_1\wedge c_2\, \in\, {\mathbb C}
\end{equation}
to the corresponding symplectic pairing on $H^1_{dR}(X,\, {\mathbb C})$.

The Jacobian $J(Y)$ of $Y$ coincides with the following quotient:
$$
(H^1_{dR}(Y,\, {\mathbb C})/H^0(Y,\, \Omega^1_Y))/H^1(Y,\, {\mathbb Z})\, =\, J(Y)\, ,
$$
and the natural principal polarization on $J(Y)$ is constructed
using the pairing in \eqref{e10} and the complex structure of $H^1_{dR}(Y,\, {\mathbb C})$.
More precisely, the holomorphic tangent bundle $TJ(Y)$ of $J(Y)$ is the trivial
holomorphic vector bundle $$J(Y)\times \overline{H^0(Y,\, \Omega_Y^1)}\, \longrightarrow\, J(Y)$$
with fiber $\overline{H^0(Y,\, \Omega^1_Y)}\,=\, H^1(Y,\, {\mathcal O}_Y)$. The Hermitian
form on $\overline{H^0(Y,\, \Omega^1_Y)}$ defined by
$$
\overline{c}_1\otimes \overline{c}_2\, \longmapsto\, -\sqrt{-1}\int_Y \overline{c}_1\wedge c_2
\, \in\, {\mathbb C}\, , \ \ c_1,\, c_2\, \in\, H^0(Y,\, \Omega^1_Y)
$$
produces the canonical principal polarization on $J(Y)$.

We noted above that the $\mathbb C$--linear isomorphism $\beta^*$ in \eqref{e9} takes
$H^1(Y,\, {\mathbb Z})$ to isomorphically $H^1(X,\, {\mathbb Z})$ and
takes the symplectic pairing in \eqref{e10} to the corresponding pairing for $X$. Since
$$
\rho\,=\, \beta^*\, ,
$$
where $\rho$ is the homomorphism in \eqref{r} (and \eqref{e8}), from the commutativity of the
diagram in \eqref{e8} we now conclude that $J(Y)$ is isomorphic to the Jacobian
$J(X)$ of $X$ as a principally polarized abelian variety. Now from the standard
Torelli theorem (see \cite[Ch.~VI, \S~3, pp. 245--246]{ACGH}) we conclude
that $X$ is isomorphic to $Y$.
\end{proof}

\subsection{Neutral Tannakian category for a compact complex torus}

Let $\mathbb T$ be a compact complex torus of complex dimension $d$. The group scheme
$\varpi({\mathbb T},\, x_0)$ is abelian because $\pi_1({\mathbb T},\, x_0)$ is so.
Hence the quotient ${\Theta}(X,\, x_0)$ of $\varpi({\mathbb T},\, x_0)$ is also abelian. As in the proof
of Theorem \ref{thm1} we consider 
the corresponding additive character spaces $\mathbb X_a(\varpi({\mathbb T},\, x_0))$ and
$\mathbb X_a(\Theta ({\mathbb T},\, x_0))$ and their Lie algebras. From Lemma \ref{01.05.2020--1} we have
$$
\mathbb X_a( \varpi({\mathbb T},\, x_0) )\,=\, H^1_{dR}({\mathbb T},\, {\mathbb C})
$$
and
$$
\mathbb X_a(\Theta ({\mathbb T},\, x_0) )\,=\, H^0({\mathbb T},\, \Omega^1_{\mathbb T})\, ,
$$
where $\Omega^1_{\mathbb T}$ is the holomorphic cotangent bundle of $\mathbb T$. Consider the linear map 
\begin{equation}\label{psi}
\Psi\, :\, H^0({\mathbb T},\, \Omega^1_{\mathbb T})\,=\, \mathbb X_a( \Theta({\mathbb T},\, x_0) )
\, \longrightarrow\, \mathbb X_a(\varpi({\mathbb T},\, x_0) )\,=\,
H^1_{dR}({\mathbb T},\, {\mathbb C})
\end{equation}
induced by the homomorphism $\mathbf{q}_{\mathbb T}$ in \eqref{e6}. 
 We note that $\Psi$ coincides with the
natural inclusion of $H^0({\mathbb T},\, \Omega^1_{\mathbb T})$ in $H^1_{dR}({\mathbb T},\, {\mathbb C})$
given by the fact that any holomorphic $1$--form on $\mathbb T$ is actually closed.

Let ${\mathbb S}$ be a compact complex torus of complex dimension $d$, and let
$$
\varphi\, :\, {\mathbb T}\, \longrightarrow\, {\mathbb S}
$$
be a diffeomorphism. Let
\begin{equation}\label{vpl}
\varphi_\natural\, :\, \varpi({\mathbb T},\, x_0)\, \longrightarrow\,
\varpi({\mathbb S},\, \varphi(x_0))
\end{equation}
be the homomorphism corresponding to $\varphi$. We say that $\varphi_\natural$ descends to a
homomorphism from ${\Theta}({\mathbb T},\, x_0)$ to ${\Theta}({\mathbb S},\, \varphi(x_0))$ if there
is a homomorphism
$$
\varphi'_\natural\, :\, {\Theta}({\mathbb T},\, x_0) \, \longrightarrow\,{\Theta}({\mathbb S},\, \varphi(x_0))
$$
such that the diagram
\begin{equation}\label{da}
\begin{matrix}
\varpi({\mathbb T},\, x_0)& \stackrel{\varphi_\natural}{\longrightarrow}& \varpi({\mathbb S},\, \varphi(x_0))\\
\mathbf{q}_{\mathbb T} \Big\downarrow\,\,\,\,\,\,\, && \mathbf{q}_{\mathbb S}\Big\downarrow\,\,\,\,\,\,\, \\
{\Theta}({\mathbb T},\, x_0)&\stackrel{\varphi'_\natural}{\longrightarrow}& {\Theta}({\mathbb S},\,\varphi(x_0))
\end{matrix}
\end{equation}
is commutative, where $\mathbf{q}_{\mathbb T}$ and $\mathbf{q}_{\mathbb S}$ are the projections
in \eqref{e6}.

\begin{proposition}\label{prop2}
Let $\varphi\, :\, {\mathbb T}\, \longrightarrow\, {\mathbb S}$ be a diffeomorphism such that
the homomorphism $\varphi_\natural$ in \eqref{vpl} descends to a
homomorphism from ${\Theta}({\mathbb T},\, x_0)$ to ${\Theta}({\mathbb S},\, \varphi(x_0))$.
Then there is a biholomorphism
$$
\widetilde{\varphi}\, :\, {\mathbb T}\, \longrightarrow\, {\mathbb S}
$$
which is homotopic to the map $\varphi$.
\end{proposition}

\begin{proof}
Let $\varphi'_\natural\, :\, {\Theta}({\mathbb T},\, x_0) \, \longrightarrow\,{\Theta}({\mathbb S},\, \varphi(x_0))$
be the homomorphism such the diagram in \eqref{da} is commutative. Let
\begin{equation}\label{e8b}
\begin{matrix}
H^0({\mathbb S},\, \Omega^1_{\mathbb S}) & = &
\mathbb X_a( \Theta({\mathbb S},\, \varphi(x_0)) )& {\longrightarrow} &
\mathbb X_a( \Theta ({\mathbb T},\, x_0)) & = & H^0({\mathbb T},\, \Omega^1_{\mathbb T})\\
&& \xi^2\Big\downarrow\,\,\,\, &&\xi^1 \Big\downarrow\,\,\,\,\\
H^1_{dR}({\mathbb S},\,{\mathbb C}) & = &
\mathbb X_a( \varpi({\mathbb S},\, \varphi(x_0)) )& \stackrel{\delta}{\longrightarrow} &
\mathbb X_a( \varpi({\mathbb T},\, x_0) ) & = & H^1_{dR}({\mathbb T},\,{\mathbb C})
\end{matrix}
\end{equation}
be the commutative diagram of $\mathbb C$--linear maps corresponding to \eqref{da}.

Let
\begin{equation}\label{vps}
\varphi^*\, :\, H^1_{dR}({\mathbb S},\, {\mathbb C})\, \longrightarrow\, H^1_{dR}({\mathbb T},\, {\mathbb C}),\,
\ \ c\, \longmapsto\, \varphi^*c
\end{equation}
be the pullback map. The homomorphism $\delta$ in \eqref{e8b} coincides with the homomorphism $\varphi^*$
in \eqref{vps}. Therefore, from \eqref{e8b} we conclude that
\begin{equation}\label{e8c}
\varphi^* (H^0({\mathbb S},\, \Omega^1_{\mathbb S}))\,=\, H^0({\mathbb T},\, \Omega^1_{\mathbb T})\, .
\end{equation}
We also have
\begin{equation}\label{e8d}
\varphi^* (H^1({\mathbb S},\, {\mathbb Z}))\,=\, H^1({\mathbb T},\, {\mathbb Z})
\end{equation}
because $\varphi$ is a diffeomorphism.

We consider $\mathbb T$ (respectively, $\mathbb S$) as a complex abelian Lie group by taking
$x_0$ (respectively, $\varphi(x_0)$) to be the identity element of $\mathbb T$ (respectively,
$\mathbb S$). From \eqref{e8c} and \eqref{e8d} it follows immediately that the homomorphism
$\varphi^*$ in \eqref{vps} induces a holomorphic isomorphism
$$
\widetilde{\varphi}^\vee\, :\, {\mathbb S}^\vee \, :=\,(H^1_{dR}({\mathbb S},\,
{\mathbb C})/H^0({\mathbb S},\, \Omega^1_{\mathbb S}))/H^1({\mathbb S},\, {\mathbb Z})
$$
$$
\longrightarrow\, (H^1_{dR}({\mathbb T},\,
{\mathbb C})/H^0({\mathbb T},\, \Omega^1_{\mathbb T}))/H^1({\mathbb T},\, {\mathbb Z})\,=\, {\mathbb T}^\vee\, ;
$$
here ${\mathbb S}^\vee\,=\, \text{Pic}^0({\mathbb S})$ is the dual torus of $\mathbb S$, and
${\mathbb T}^\vee\,=\, \text{Pic}^0({\mathbb T})$ is the dual torus of $\mathbb T$. Let
$$
\widetilde{\varphi}\, :\, {\mathbb T}\,=\, ({\mathbb T}^\vee)^\vee \, \longrightarrow\,
({\mathbb S}^\vee)^\vee\,=\,{\mathbb S}
$$
be the dual of the above homomorphism $\widetilde{\varphi}^\vee$.

{}From the construction of the above homomorphism $\widetilde{\varphi}$ it is evident that
the pullback homomorphism
$$
\widetilde{\varphi}^*\, :\, H^1_{dR}({\mathbb S},\, {\mathbb C})\, \longrightarrow\, H^1_{dR}({\mathbb T},\, {\mathbb C}),\,
\ \ c\, \longmapsto\, \widetilde{\varphi}^*c
$$
coincides with the homomorphism $\varphi^*$ in \eqref{vps}. This implies that the
two maps $\widetilde{\varphi}$ and $\varphi$ are homotopic.
\end{proof}

\section{Completely decomposable vector bundles}\label{lastsec}

As before, $X$ is a compact connected K\"ahler manifold.
We define and study in this section a subcategory of ${\mathcal C}_{dR}(X)$ given by the following:

\begin{definition}
Let ${\mathcal D}_{dR}(X)$ be the full subcategory of ${\mathcal C}_{dR}(X)$ whose objects are
pairs $(E,\, D)$ satisfying the condition that the holomorphic vector bundle
$E$ is a direct sum of holomorphic line bundles.
\end{definition}

It is straightforward to check that ${\mathcal D}_{dR}(X)$ is stable under tensor products and duals. In order to 
prove that ${\mathcal D}_{dR}(X)$ is abelian we will need the next two lemmas.

\begin{lemma}\label{lem5a}
Let $L$ be a holomorphic line bundle on $X$. Then $L$ admits a holomorphic connection if and only
if $c_1(L)\,=\, 0$. If $L$ admits a holomorphic connection, then any holomorphic connection on $L$
is integrable.

Let $L_i$, $1\, \leq\, i\, \leq\, r$, be holomorphic line bundles on $X$ such
that the holomorphic vector bundle $\bigoplus_{i=1}^r L_i$ admits a holomorphic connection.
Then each $L_i$ admits a holomorphic connection; in other words, $c_1(L_i)\,=\, 0$.
\end{lemma}

\begin{proof}
Since $X$ is K\"ahler, if $L$ admits a holomorphic connection, then $c_1(L)\,=\, 0$
\cite[pp.~192--193, Theorem 4]{At}. If $c_1(L)\,=\, 0$, then $L$ admits an integrable holomorphic
connection $\nabla^L$ whose monodromy lies in ${\rm U}(1)$ see \cite[Corollary 1.3.12, p.40]{LT}; note also that
any holomorphic line bundle admits
a Hermitian--Einstein metric \cite{UY}, \cite[p.~61, Theorem 3.0.1]{LT}.

If $L$ admits a (integrable) holomorphic connection $\nabla^L$, then any holomorphic connection on $L$ is of the form
$\nabla^L+\beta$, where $\beta\, \in\, H^0(X,\, \Omega^1_X)$. The curvature of
$\nabla^L+\beta$ is ${\rm d}\beta$, because $\nabla^L$ is integrable. Since any holomorphic $1$--form on $X$ is
closed, we conclude that $\nabla^L+\beta$ is integrable.

Let $p_j\, :\, \bigoplus_{i=1}^r L_i\, \longrightarrow\, L_j$ be the projection to the
$j$--th factor. Then for any holomorphic connection $D$ on $\bigoplus_{i=1}^r L_i$, the
composition of homomorphisms
$$
L_j\, \hookrightarrow\, \bigoplus_{i=1}^r L_i\, \stackrel{D}{\longrightarrow}\,
(\bigoplus_{i=1}^r L_i)\otimes\Omega^1_X\, \stackrel{p_j\otimes{\rm Id}_{\Omega^1_X}}{\longrightarrow}\,
L_j\otimes\Omega^1_X
$$
is a holomorphic connection on $L_j$.
\end{proof}

\begin{lemma}\label{5b}
Let $L_i$, $1\, \leq\, i\, \leq\, r$, and ${\mathcal L}_j$, $1\, \leq\, j\,\leq\, \ell$,
be holomorphic line bundles on $X$ such that $c_1(L_i)\,=\, 0\,=\, c_1({\mathcal L}_j)$
for all $i,\, j$. Let
$$
f\, :\, \bigoplus_{i=1}^r L_i\,\longrightarrow\, \bigoplus_{j=1}^\ell {\mathcal L}_j
$$
be a holomorphic homomorphism. Then the following two hold:
\begin{enumerate}
\item If ${\rm kernel}(f)\, \not=\, 0$, then ${\rm kernel}(f)$ is a direct sum
of holomorphic line bundles of vanishing first Chern class.

\item If ${\rm cokernel}(f)\, \not=\, 0$, then ${\rm cokernel}(f)$ is a direct sum
of holomorphic line bundles of vanishing first Chern class.
\end{enumerate}
\end{lemma}

\begin{proof}
Since $X$ is K\"ahler every holomorphic line bundle $L'$ on $X$ with $c_1(L')\,=\, 0$ admits a flat
unitary connection. In fact, there is a unique flat unitary connection on $L'$; the flat Hermitian structure
is unique up to a constant scalar multiplication.
Equip each $L_i$ (respectively, ${\mathcal L}_j$) with an integrable holomorphic connection
$\nabla^i$ (respectively, $\widetilde{\nabla}^j$) such that the monodromy lies in ${\rm U}(1)$.
Therefore, $\bigoplus_{i=1}^r \nabla^i$ is a flat unitary connection on $\bigoplus_{i=1}^r L_i$,
and $\bigoplus_{j=1}^\ell \widetilde{\nabla}^j$ is a flat unitary connection on 
$\bigoplus_{j=1}^\ell {\mathcal L}_j$. Since any flat unitary connection is a Hermitian--Einstein
connection with respect to any K\"ahler metric, $\bigoplus_{i=1}^r \nabla^i$ and
$\bigoplus_{j=1}^\ell \widetilde{\nabla}^j$ are Hermitian--Einstein
connections on $\bigoplus_{i=1}^r L_i$ and $\bigoplus_{j=1}^\ell {\mathcal L}_j$ respectively.

If $A$ and $B$ are holomorphic vector bundles on $X$ with $c_1(A)\,=\, 0\, =\, c_1(B)$, and 
are equipped with Hermitian--Einstein connections $\nabla^A$ and $\nabla^B$ respectively,
then every holomorphic homomorphism $f'\, :\, A\, \longrightarrow\, B$ is flat with respect to
the connection on the holomorphic vector bundle $\text{Hom}(A,\, B)$ induced by $\nabla^A$ and $\nabla^B$
(see \cite[p.~50, Theorem 2.2.1]{LT}). Therefore, we conclude that the
holomorphic homomorphism $f$ in the lemma intertwines the connections $\bigoplus_{i=1}^r\nabla^i$ 
on $\bigoplus_{i=1}^r L_i$ and $\bigoplus_{j=1}^\ell \widetilde{\nabla}^j$ on
$\bigoplus_{j=1}^\ell {\mathcal L}_j$ (equivalently, it is flat with respect to
the connection on $\text{Hom}(\bigoplus_{i=1}^r L_i,\, \bigoplus_{j=1}^\ell {\mathcal L}_j)$
induced by $\bigoplus_{i=1}^r \nabla^i$ and $\bigoplus_{j=1}^\ell \widetilde{\nabla}^j$). Consequently,
${\rm kernel}(f)$ is preserved by the connection $\bigoplus_{i=1}^r\nabla^i$
on $\bigoplus_{i=1}^r L_i$ and
${\rm image}(f)$ is preserved by the connection $\bigoplus_{j=1}^\ell \widetilde{\nabla}^j$
on $\bigoplus_{j=1}^\ell {\mathcal L}_j$. From Remark
\ref{remb} we know that both ${\rm kernel}(f)$ and ${\rm image}(f)$ are locally free.
Therefore, ${\rm kernel}(f)$ is a holomorphic subbundle of $\bigoplus_{i=1}^r L_i$,
and ${\rm cokernel}(f)$ is a quotient bundle of
$\bigoplus_{j=1}^\ell {\mathcal L}_j$. Note that ${\rm cokernel}(f)$ has an integrable holomorphic
connection induced by $\bigoplus_{j=1}^\ell \widetilde{\nabla}^j$, because
${\rm image}(f)$ is preserved by $\bigoplus_{j=1}^\ell \widetilde{\nabla}^j$.

Since ${\rm kernel}(f)$ and ${\rm cokernel}(f)$ both admit integrable holomorphic
connections, we have
$$
c_1({\rm kernel}(f))\,=\,0\, =\, c_1({\rm cokernel}(f))
$$
(see \cite[pp.~192--193, Theorem 4]{At}), and hence we conclude that
\begin{equation}\label{ed}
\text{degree}({\rm kernel}(f))\,=\,0\, =\, \text{degree}({\rm cokernel}(f))
\end{equation}
with respect to any K\"ahler form on $X$.

Any holomorphic line bundle on $X$ is stable. Since $\bigoplus_{i=1}^r L_i$ is a direct sum
of stable bundles of slope zero, we conclude that
$\bigoplus_{i=1}^r L_i$ is a polystable vector bundle. Any subbundle of degree zero of a polystable vector
bundle of degree zero is again polystable. Therefore, from \eqref{ed} we conclude that
${\rm kernel}(f)$ is a polystable vector bundle of degree zero, if
${\rm kernel}(f)$ is nonzero. Assume that ${\rm kernel}(f)$ is nonzero, so
${\rm kernel}(f)$ is a polystable vector bundle of degree zero.
Hence ${\rm kernel}(f)$ is a direct sum
of stable vector bundle of degree zero. If $S$ is a stable subbundle of ${\rm kernel}(f)$ of degree zero,
then $S$ is also a stable subbundle of $\bigoplus_{i=1}^r L_i$ of degree zero, because ${\rm kernel}(f)$
is a subbundle of $\bigoplus_{i=1}^r L_i$. But any stable subbundle of $\bigoplus_{i=1}^r L_i$ of degree zero
is holomorphically isomorphic to $L_i$ for some $1\, \leq\, i\, \leq\, r$. Therefore, we conclude that
${\rm kernel}(f)$ is a direct sum
of holomorphic line bundles of vanishing first Chern class, if ${\rm kernel}(f)\, \not=\, 0$.

Next we note that $\bigoplus_{j=1}^\ell {\mathcal L}_j$ is a polystable vector bundle of degree zero,
and ${\rm cokernel}(f)$ is a quotient bundle of $\bigoplus_{j=1}^\ell {\mathcal L}_j$ of degree zero.
Hence ${\rm cokernel}(f)$ is a polystable vector bundle of degree zero, if
${\rm cokernel}(f)$ is nonzero. Assume that ${\rm cokernel}(f)$ is nonzero. So
${\rm cokernel}(f)$ is a polystable vector bundle of degree zero. Therefore,
${\rm cokernel}(f)$ is a direct sum of stable vector bundles of degree zero.
Any stable quotient of
${\rm cokernel}(f)$ of degree zero is also a stable quotient of
$\bigoplus_{j=1}^\ell {\mathcal L}_j$ of degree zero, because ${\rm cokernel}(f)$ is a quotient 
of $\bigoplus_{j=1}^\ell {\mathcal L}_j$. But any stable quotient of
$\bigoplus_{j=1}^\ell {\mathcal L}_j$ of degree zero is holomorphically isomorphic to ${\mathcal L}_j$
for some $1\, \leq\, j\, \leq\, \ell$. Therefore, if ${\rm cokernel}(f)\, \not=\, 0$, then
${\rm cokernel}(f)$ is a direct sum of holomorphic line bundles of vanishing first Chern class.
\end{proof}

Using Lemma \ref{5b} it can be deduced that the category ${\mathcal D}_{dR}(X)$ is abelian.

Fix a point $x_0\, \in\, X$. Equip the category ${\mathcal D}_{dR}(X)$ with the exact
functor to the category of finite dimensional complex vector spaces defined by
$(E,\,D)\,\longmapsto\, E|_{x_0}$. Let
\begin{equation}\label{e2l}
{\Delta}(X,\, x_0)
\end{equation}
be the affine group scheme over $\mathbb C$ corresponding
to ${\mathcal D}_{dR}(X)$ \cite[p.~130, Theorem 2.11]{DMOS}.
Note that $\Delta(X,\,x_0)$ is the target of a morphism of affine group schemes 
\begin{equation}\label{e6l}
\varpi(X,\, x_0)\, \longrightarrow\, {\Delta}(X,\, x_0)\, ,
\end{equation}
where $\varpi(X,\, x_0)$ is the group scheme in \eqref{e5}.
According to the standard criterion \cite[Proposition 2.21, p.~139]{DMOS}, amplified
by \cite[Lemma 2.1]{biswas-hai-dos_santos19}, using Lemma \ref{5b} it follows that
the homomorphism in \eqref{e6l} is a quotient morphism.

The homomorphism $ \mathbf{q}_X$ in (\ref{e6}) factors through the quotient ${\Delta}(X,\, x_0)$ and gives a 
quotient homomorphism
\begin{equation}\label{e6m}
\mathbf{p}_X\, :\, {\Delta}(X,\, x_0)\, \longrightarrow\, {\Theta}(X,\, x_0)\,.
\end{equation}

For a pointed compact K\"ahler manifold $(Y,\,y_0)$ it can be shown that the natural
homomorphism
$$
\Delta(X,\,x_0)\times\Delta(Y,y_0)\,\longrightarrow
\,\Delta(X\times Y,\, (x_0,\,y_0))
$$
is an isomorphism; it's proof is similar to that of Proposition \ref{prop1}.

The homomorphism of additive character $\mathbb C$--vector spaces
$$
\mathbb X_a({{\Theta}(X,\, x_0)})\, \longrightarrow\, \mathbb X_a({{\Delta}(X,\, x_0)})
$$
induced by $\mathbf{p}_X$ in \eqref{e6m} is an isomorphism as follows from Lemma \ref{5b} and the interpretation of these vector spaces as extension groups (see Section \ref{01.05.2020--2}). In view of this,
examining the proofs of Theorem \ref{thm1} and Proposition \ref{prop2} we
conclude that these two results remain valid if ${\Theta}(X,\, x_0)$ is replaced
by the group scheme ${\Delta}(X,\, x_0)$ in \eqref{e2l}.

\section{Connections on pseudostable vector bundles}\label{07.05.2020--1}

We shall now relate the group scheme $\Theta(X,\, x_0)$ of Section 
\ref{03.05.2020--4} to another group scheme which deals only with a full subcategory of vector bundles.

\begin{definition}\label{03.05.2020--3}
A holomorphic vector bundle $E$ over $X$ is said to be pseudostable (see \cite{BG}) if the Higgs vector
bundle $(E,\,0)$ is an object of the category ${\mathcal C}_{Dol}(X)$ introduced in Section \ref{se2}. In other
words, $E$ is pseudostable if and only if 
\begin{itemize}
\item ${\rm degree}(E)\,=\, 0$,

\item$ch_2(E)\wedge\omega^{d-2}\,=\,0$, and 

\item there exists a filtration of $E$ by holomorphic {\it subbundles}
\[
0\,=\,E_0\,\subset\, E_1\, \subset\, \cdots \,\subset\, E_{\ell-1}\, \subset\, E_\ell\,=\,E
\] 
in which every $E_j/E_{j-1}$, $1\leq\, j\, \leq\, \ell$, is stable and of degree zero.
\end{itemize}
(Actually the third condition implies the first condition.)
The full subcategory of pseudostable vector bundles on $X$ will be denoted by ${\rm Vect}^{\rm s}_0(X)$.
\end{definition}

The following theorem, which is based on the works of Simpson and Corlette,
is proved in \cite{Si2} (see \cite[pp.~35--37, (3.4.1)--(3.4.5)]{Si2} and
\cite[p.~70]{Si2}).

\begin{theorem}[{\cite[(3.4.1)--(3.4.5)]{Si2}}]\label{04.05.2020--3}\mbox{}
The following three statements hold.
\begin{enumerate}
\item The category ${\rm Vect}^{\rm s}_0(X)$ is abelian and stable under the tensor product of vector bundles.

\item Any $E\,\in\,{\rm Vect}_0^{\rm s}(X)$ carries a canonical holomorphic integrable connection. 

\item If $X$ is projective, and the cohomology class of $\omega$ is rational,
then ${\rm Vect}^{\rm s}_0(X)$ coincides with the category of all semistable 
vector bundles $E$ such that $$ch_1(E)\wedge\omega^{d-1}\,=\,ch_2(E)\wedge\omega^{d-2}\,=\,0$$ (see
\cite[p.~39, Theorem 2]{Si2} for it).
\end{enumerate}
\end{theorem}

\begin{remark}\label{rem1}
As mentioned earlier in Section \ref{se2}, the category ${\mathcal C}_{Dol}(X)$ is equivalent to
${\mathcal C}_{dR}(X)$ \cite[p.~36, Lemma 3.5]{Si2}. Take $(E,\,\theta)\, \in\, {\mathcal C}_{Dol}(X)$,
and let $(V,\, \nabla)\, \in\, {\mathcal C}_{dR}(X)$ be the object corresponding to $(E,\,\theta)$ by the
equivalence of categories between ${\mathcal C}_{Dol}(X)$ and ${\mathcal C}_{dR}(X)$. Although the $C^\infty$ vector
bundles underlying $E$ and $V$ coincide, their holomorphic structure do not coincide in general. However,
if $\theta\,=\, 0$, then the holomorphic structures of $E$ and $V$ coincide. Therefore, $\nabla$ is
a holomorphic integrable connection on $V\,=\, E$. We note that if $E$ is polystable, then this holomorphic
connection $\nabla$ on $E$ coincides with the unique holomorphic integrable connection on $E$ whose
monodromy is unitary. See \cite[p.~4004, Theorem 3.1]{BS} for a generalization
of this canonical connection on $E\, \in\,{\rm Vect}^{\rm s}_0(X)$ to the principal bundles over $X$.
\end{remark}

\begin{remark}\label{rem2}
Take $E_1,\, E_2\, \in\, {\rm Vect}^{\rm s}_0(X)$. The canonical holomorphic integrable connections
on $E_1$ and $E_2$ (see Remark \ref{rem1}) will be denoted by $\nabla^1$ and $\nabla^2$ respectively. Let
$$
\Phi\, :\, E_1\, \longrightarrow\, E_2
$$
be a homomorphism of coherent analytic sheaves. Then $\Phi$ is a homomorphism
from $(E_1,\, 0)\, \in\, {\mathcal C}_{Dol}(X)$ to $(E_2,\, 0)\, \in\, {\mathcal C}_{Dol}(X)$.
Therefore, from the equivalence of categories between ${\mathcal C}_{Dol}(X)$ and
${\mathcal C}_{dR}(X)$
(\cite[p.~36, Lemma 3.5]{Si2}) we conclude that $\Phi$ is a homomorphism from
$(E_1, \, \nabla^1)\, \in\, {\mathcal C}_{dR}(X)$ to $(E_2, \, \nabla^2)\, \in\, {\mathcal C}_{dR}(X)$.
In particular, $\Phi(E_1)\, \subset\, E_2$ is preserved by the
integrable holomorphic connection $\nabla^2$ on $E_2$. Now
setting $E_1$ to be the trivial line bundle ${\mathcal O}_X$ we conclude that every holomorphic section
$s\, \in\, H^0(X,\, E_2)$ is flat (same as integrable) with respect to the
integrable holomorphic connection $\nabla^2$ on $E_2$.
It also follows that the canonical connection on $E_1\oplus E_2$ coincides with $\nabla^1\oplus \nabla^2$.
Moreover, the canonical connection on $E_1\otimes E_2$ (respectively, $\text{Hom}(E_1, \, E_2)$)
coincides with the connection on $E_1\otimes E_2$ (respectively, $\text{Hom}(E_1, \, E_2)$) induced
by $\nabla^1$ and $\nabla^2$. Also, the canonical connection on the dual vector bundle $E^*_1$ coincides
with the one induced by the connection $\nabla^1$ on $E_1$.
These properties of the canonical connection were crucial in the
proofs of \cite[p.~4004, Theorem 3.1]{BS} and \cite[p.~20, Theorem 1.1]{BG}.
\end{remark}

It then follows that there exists an affine group scheme $\pi^{\rm S}(X,\,x_0)$ over
$\mathbb C$ such that the functor 
\[
\bullet|_{x_0} \,:\, \mathrm{Vect}_0^{\rm s}(X)\,\longrightarrow\, \mathrm{Vect}_{\mathbb C}
\]
induces an equivalence between $ \mathrm{Vect}_0^{\rm s}(X)$ and the
category of finite dimensional algebraic representations of $\pi^{\rm S}(X,\, x_0)$. 

Following the path taken in the previous sections, we define ${\mathcal S}_{dR}(X)$ as the full
subcategory of ${\mathcal C}_{dR}(X)$ consisting of those $(E,\,\nabla)$ such that $E$ belongs to
${\rm Vect}_0^{\rm s}(X)$. This produces another group scheme
\begin{equation}\label{03.05.2020--5}
\Sigma(X,\,x_0)
\end{equation}
whose category of representations is, via the functor $\bullet|_{x_0}$, simply $\mathcal S_{dR}(X)$.

Using the forgetful functor
\[
U\,:\, \mathcal S_{dR}(X)\, \longrightarrow\, {\rm Vect}_0^{\rm s}(X)
\] 
and the natural inclusion
\[
J\,:\,\mathcal T_{dR}(X)\,\longrightarrow\,{\mathcal S}_{dR}(X),
\] 
we arrive at morphisms
\begin{equation}\label{03.05.2020--6}
\pi^{\rm S}(X,\,x_0)\,\stackrel{\mathbf u_X}{\longrightarrow}\,\Sigma(X,\,x_0)\,
\stackrel{\mathbf j_X}{\longrightarrow}\, \Theta(X,\, x_0)\, .
\end{equation}
Clearly, $\mathbf j_X\circ\mathbf u_X$ is the trivial homomorphism.

We shall now explore some properties of ${\mathbf u}_X$. For that, we make a group theoretical interlude. No 
particular property concerning the ground field $\mathbb C$ is required so, in what follows, an ``affine group 
scheme'' should be interpreted as an ``affine group scheme over a unspecified field.''

From the existence of limits in the category of affine group schemes, we can make: 

\begin{definition}
Let $G$ be an affine group scheme and $H\,\subset\, G$ a closed subgroup scheme. The normal closure of $H$ inside
$G$ is the intersection of all normal and closed subgroup schemes of $G$ containing $H$.
\end{definition}

Here are two simple features of the normal closure. The reader unfamiliar with the construction of arbitrary 
quotients in the theory of group schemes should consult chapters 15 and 16 of \cite{waterhouse79}. In particular, 
the verification of the following result is immediate once 15.4 and 16.3 in \cite{waterhouse79} are understood.

\begin{lemma}\label{03.09.2020--1}
Let $G$ be an affine group scheme, $H\,\subset\, G$ a closed subgroup scheme and $N$ the normal closure of $H$ in $G$.
\begin{enumerate}
\item
For each quotient morphism $q\,:\,G\,\longrightarrow\, Q$ of group schemes such that $q|_H$ is the trivial
morphism, there exists
a factorization 
\[
\xymatrix{
G\ar[r]^q\ar[d]& Q
\\
G/N\ar[ru]_{\overline q}.
}
\]

\item Let $D\,\subset\, G$ be a closed and normal subgroup scheme containing $H$. Suppose that $D$ enjoys
the property attributed to $N$ in the previous item. Then $D\,=\,N$.
\end{enumerate}
\end{lemma}

\begin{theorem}\label{04.05.2020--5}
The following three statements hold.
\begin{enumerate}
\item In \eqref{03.05.2020--6},
the morphism $\mathbf{u}_X$ is a closed immersion, and $\mathbf j_X$ is a quotient map. 

\item The normal closure of ${\rm Im}(\mathbf{u}_X)$ inside $\Sigma(X,\,x_0)$ is ${\rm Ker}(\mathbf j_X)$. 

\item The morphism $\mathbf u_X$ possesses a left inverse
\[
{\mathbf v}_X\,:\,\Sigma(X,\,x_0)\,\longrightarrow\,\pi^{\rm S}(X,\,x_0)\, .
\]
\end{enumerate}
\end{theorem}

\begin{proof}
The morphism $$\mathbf q_X\,:\,\varpi(X,\,x_0)\,\longrightarrow\,\Theta(X,\,x_0)$$ in \eqref{e6} factors through 
$\mathbf j_X$ so that the fact that ${\mathbf q}_X$ is a quotient morphism shows that ${\mathbf j}_X$ is likewise. 
That ${\mathbf u}_X$ is a closed immersion follows from the standard criterion \cite[p.~139, Proposition
2.21]{DMOS} and Theorem \ref{04.05.2020--3}.

To show that $\mathrm{Ker}({\mathbf q}_X)$ is the normal closure of ${\rm Im}({\mathbf u}_X)$
inside $\Sigma(X,\,x_0)$, we verify
condition (2) in Lemma \ref{03.09.2020--1}. Let $$\rho\,:\,\Sigma(X,\,x_0)\,\longrightarrow\, G$$ be a quotient
morphism such that $\rho\circ{\bf u}_X$ is trivial. 
Let \[\xi\,:\,{\rm Rep}_{\mathbb C}(\Sigma(X,\,x_0))\,\longrightarrow\,\mathcal S_{dR}(X)
\]
be an inverse to the tensor equivalence $\bullet|_{x_0}$ \cite[I.4.4]{Sa} and let
${\rm Res}_\rho$ be the functor restricting representations from of $G$ to $\Sigma(X,\,x_0)$. It follows that,
for each $V\,\in\, {\rm Rep}_{\mathbb C}(G)$, the vector bundle underlying $\xi\circ{\rm Res}_{\rho}(V)$ is
trivial, which means that $\xi\circ{\rm Res}_{\rho}$ is in fact a functor to $\mathcal T_{dR}(X)$. Passing
to group schemes we obtain a morphism $\Theta(X,\,x_0)\,\longrightarrow\, G$ factoring $\rho$ \cite[II.3.3.1, 148ff]{Sa}. 

Assigning to any pseudostable vector bundle $E\, \in\, {\rm Vect}_0^{\rm s}(X)$ the canonical
integrable holomorphic connection on $E$ (see Theorem \ref{04.05.2020--3}(2)), a functor is obtained; call it 
$${\mathbb V}\,:\,\mathrm{Vect}_0^{\rm s}(X)\,\longrightarrow\,\mathcal S_{dR}(X)\, .$$
This $\mathbb V$ is a tensor functor (see Remark \ref{rem2}). Hence ${\mathbb V}$ produces the 
section $\mathbf v_X$ to $\mathbf u_X$.
\end{proof}

\begin{proposition}Suppose that $X$ is a compact Riemann surface of genus at least two. Then
${\rm Im}(\mathbf{u}_X)$ is not normal in $\Sigma(X,\,x_0)$.
\end{proposition}

\begin{proof}
If ${\rm Im}(\mathbf{u}_X)$ is normal, then, for each representation $$\rho:\Sigma(X,\,x_0)\,\longrightarrow\,
{\rm GL}(V)\, ,$$ the subspace $V^{\pi^S(X,\,x_0)}$ is invariant under $\Sigma(X,\,x_0)$. This translates into the
following property on the side of connections. Given $(E,\,\nabla)\,\in\, {\mathcal S}_{dR}(X)$, the
largest trivial subbundle of $E$ is preserved by $\nabla$.

Let $L$ be a non-trivial invertible sheaf affording an integrable holomorphic connection
$$D\,:\,L\,\longrightarrow\, L\otimes\Omega_X^1\, .$$ In addition, let us choose $L$ satisfying
$H^0(X,\, L\otimes\Omega_X^1)\,\not=\,0$ (this is possible given the hypothesis on the genus).
Let $$\varphi\,\in\, H^0(X,\, L\otimes\Omega_X^1)\setminus\{0\}$$ and consider the connection $\nabla$ on
$\mathcal O_X\oplus L$ defined, on an unspecified open subset, by
\[\nabla(a,\,\ell)\,=\,(da,\,a\varphi+D\ell)\, . 
\] 
Then, ${\mathcal O}_X\oplus\{0\}\,\subset\,\mathcal O_X\oplus L$ is the largest trivial subbundle, and
$\nabla(1,\,0)\,=\,(0,\,\varphi)$, meaning that $\mathcal O_X\oplus\{0\}$ is not preserved by $\nabla$.
\end{proof}

\begin{remark}\label{rem3}
Let $\widetilde{{\rm Vect^{\rm s}}}_0(X)$ be the full subcategory of
${\rm Vect}_0^{\rm s}(X)$ defined by all holomorphic vector bundles
in ${\rm Vect}_0^{\rm s}(X)$ that are direct sum of holomorphic line bundles. It is a Tannakian
subcategory, and defines a quotient group scheme
$$
\pi^{\rm S}(X,\,x_0)\, \longrightarrow\, \widetilde{\pi}^{\rm S}(X,\, x_0)
$$
of $\pi^{\rm S}(X,\,x_0)$. A result similar to Theorem \ref{04.05.2020--5} can be 
deduced by replacing $\Sigma$ with $\Delta$ (constructed in \eqref{e2l}) and $\pi^S$ with $\widetilde{\pi}^{\rm S}$.
\end{remark}

\section*{Acknowledgements}

We thank the referees for their helpful comments.
We thank Carlos Simpson for a useful discussion. The first and third named authors thank
International Center for Theoretical Sciences, Bangalore, for hospitality; they were also
partially supported by the French government through the UCAJEDI Investments in the 
Future project managed by the National Research Agency (ANR) with the reference number 
ANR2152IDEX201. 
The first-named author is partially supported by a J. C. Bose Fellowship, and
school of mathematics, TIFR, is supported by 12-R$\&$D-TFR-5.01-0500.

\end{document}